\documentclass[11pt]{amsart}
\usepackage{amscd}
\usepackage{amssymb}
\usepackage{hyperref}
\usepackage{mathrsfs}
\usepackage[dvips]{graphics}
\usepackage{epsfig,epsf}
\usepackage{verbatim}

\input{xy}
\xyoption{all}

\newcommand\AAA{\mathbb{A}}

\newcommand\CC{\mathbb{C}}

\newcommand\NN{\mathbb{N}}

\newcommand\PP{\mathbb{P}}

\newcommand\RR{\mathbb{R}}

\newcommand\ZZ{\mathbb{Z}}

\newcommand\bD{{\mathbf{D}}}

\newcommand\bK{{\mathbf{K}}}

\newcommand{\rd}{\partial}

\newcommand\thalf{{\textstyle{\frac{1}{2}}}}

\newcommand\half{{{\frac{1}{2}}}}

\newcommand\fs{{\mathfrak{s}}}

\newcommand\ft{{\mathfrak{t}}}

\newcommand\eps{\varepsilon}

\newcommand\la{\lambda}
\newcommand\La{\Lambda}

\newcommand\si{\sigma}
\newcommand\Si{\Sigma}

\newcommand\SO{\operatorname{SO}}

\newcommand\Hom{\operatorname{Hom}}

\newcommand\Ker{\operatorname{Ker}}

\newcommand\PD{\operatorname{PD}}

\newcommand\SW{SW}
\newcommand\Sym{\operatorname{Sym}}

\newcommand\even{{\mathrm{even}}}

\newcommand\spinc{\text{$\text{spin}^c$ }}

\newcommand\Spinc{\text{$\text{Spin}^c$}}

\newcommand\sM{{\mathscr{M}}}

\newtheorem{thm}{Theorem}[section]
\newtheorem{lem}[thm]{Lemma}

\newtheorem{cor}[thm]{Corollary}

\newtheorem{prop}[thm]{Proposition}

\theoremstyle{definition}

\renewcommand{\thecase}{}

\newtheorem{conj}[thm]{Conjecture}

\newtheorem{rmk}[thm]{Remark}
\renewcommand{\thestep}{}

\theoremstyle{remark}

\makeatletter
\def\alphenumi{
  \def\theenumi{\alph{enumi}}
  \def\p@enumi{\theenumi}
  \def\labelenumi{(\@alph\c@enumi)}}
\makeatother

\makeatletter
\def\thecase{\@arabic\c@case}
\makeatother

\numberwithin{equation}{section}

\makeatletter
\def\thestep{\@arabic\c@step}
\makeatother

\newtheorem{hyp}[thm]{Hypothesis}

\begin{document}
\title[Superconformal simple type and Witten's conjecture]
{Superconformal simple type and Witten's conjecture}
\author[Paul M. N. Feehan]{Paul M. N. Feehan}
\address{Department of Mathematics\\
Rutgers, The State University of New Jersey\\
Piscataway, NJ 08854-8019}
\email{feehan@math.rutgers.edu}
\urladdr{math.rutgers.edu/$\sim$feehan}
\author[Thomas G. Leness]{Thomas G. Leness}
\address{Department of Mathematics\\
Florida International University\\
Miami, FL 33199}
\email{lenesst@fiu.edu}
\urladdr{fiu.edu/$\sim$lenesst}
\dedicatory{}
\thanks{Paul Feehan was partially supported by National Science Foundation grant DMS-1510064 and Thomas Leness was partially supported by National Science Foundation grant DMS-1510063.}
\begin{abstract}
Let $X$ be a smooth, closed, connected, orientable four-manifold with
$b^1(X)=0$ and $b^+(X)\ge 3$ and odd.
We show that if $X$ has
Seiberg--Witten simple type, then the
$\SO(3)$-monopole cobordism formula of \cite{FL5} implies Witten's Conjecture
relating the Donaldson and Seiberg--Witten invariants.
\end{abstract}

\date{This version: September 25, 2019, incorporating final galley proof corrections. Advances in Mathematics (2019), https://doi.org/10.1016/j.aim.2019.106821}
\maketitle

\section{Introduction}
For a closed four-manifold $X$ we will use the characteristic
numbers,
\begin{equation}
  \label{eq:CharNumbers}
  \begin{aligned}
    c_1^2(X) &:= 2e(X)+3\si(X),
    \\
    \chi_h(X) &:=(e(X)+\si(X))/4,
    \\
c(X)&:=\chi_h(X)-c_1^2(X),
\end{aligned}
\end{equation}
where $e(X)$ and $\si(X)$ are the Euler characteristic and signature of $X$.
We call a four-manifold {\em standard\/} if it is closed, connected, oriented, and smooth
with $b^+(X)\ge 3$ and odd and $b^1(X)=0$.
For a standard four-manifold, the Seiberg--Witten invariants define a function, $\SW_X:\Spinc(X)\to\ZZ$,
on the set of $\spinc$ structures on $X$.
The {\em Seiberg--Witten basic classes\/}, $B(X)$, are
the image under $c_1:\Spinc(X)\to H^2(X;\ZZ)$ of the support of $\SW_X$.
The manifold $X$ has {\em Seiberg--Witten simple type\/} if $K^2=c_1^2(X)$ for all $K\in B(X)$.
Further definitions of
and notations for the Donaldson and Seiberg--Witten invariants appear in
 \S \ref{subsec:SWDefn} and \S \ref{subsec:DonDefn}.

\begin{conj}[Witten's conjecture]
\label{conj:WC}
Let $X$ be a standard four-manifold.  If $X$ has Seiberg--Witten simple type, then
$X$ has Kronheimer--Mrowka simple type, the
Seiberg--Witten and Kronheimer--Mrowka basic classes coincide, and for any $w\in H^2(X;\ZZ)$
and $h\in H_2(X;\RR)$ the
Donaldson invariants satisfy
\begin{equation}
\label{eq:WConjecture}
\bD^{w}_X(h)
=
2^{2-(\chi_h-c_1^2)}e^{Q_X(h)/2}
\sum_{\fs\in\Spinc(X)}(-1)^{\half(w^{2}+c_1(\fs)\cdot w)}
SW_X(\fs)e^{\langle c_1(\fs),h\rangle}.
\end{equation}
\end{conj}

As defined by Mari\~no, Moore, and Peradze \cite{MMPhep,MMPdg},
the manifold $X$ has {\em superconformal simple type\/} if
$c(X)\le 3$ or $c(X)\ge 4$ and for $w\in H^2(X;\ZZ)$ characteristic,
\begin{equation}
\label{eq:SWPolynomial}
\SW_X^{w,i}(h):=
\sum_{\fs\in\Spinc(X)}(-1)^{\half(w^{2}+c_1(\fs)\cdot w)}
SW_X(\fs)\langle c_1(\fs),h\rangle^i=0
\quad\text{for $i\le c(X)-4$},
\end{equation}
and all $h\in H_2(X;\RR)$.
Our goal in this article is to prove the following

\begin{thm}
\label{thm:SCSTImpliesWC}
Let $X$ be a standard four-manifold that has superconformal simple type. Then the $\SO(3)$-monopole cobordism formula (Theorem \ref{thm:Cobordism}) implies that $X$ satisfies Witten's Conjecture \ref{conj:WC}.
\end{thm}

Combining Theorem \ref{thm:SCSTImpliesWC} with the results
of \cite{FL8} yields the following

\begin{cor}
\label{cor:NonZeroIntImpliesWC}
Let $X$ be a standard four-manifold of Seiberg--Witten simple type and assume Hypothesis \ref{hyp:Local_gluing_map_properties}. Then $X$ satisfies Witten's Conjecture \ref{conj:WC}.
\end{cor}

In \cite{FL5}, we proved the required $\SO(3)$-monopole cobordism formula, restated in this article as Theorem \ref{thm:Cobordism}, assuming the validity of certain technical properties --- comprising Hypothesis \ref{hyp:Local_gluing_map_properties} and described in more detail in Remark \ref{rmk:GluingThmProperties} --- of the local gluing maps for $\SO(3)$ monopoles constructed in \cite{FL3}. A proof of the required local $\SO(3)$-monopole gluing-map properties, which may be expected from known properties of local gluing maps for anti-self-dual $\SO(3)$ connections and Seiberg--Witten monopoles, is currently being developed by the authors \cite{Feehan_Leness_monopolegluingbook}.
However, Theorem \ref{thm:SCSTImpliesWC} is a direct consequence of the $\SO(3)$-monopole cobordism formula, Theorem \ref{thm:Cobordism}.

One might draw a comparison between our use of the $\SO(3)$-monopole cobordism formula in our proof of Theorem \ref{thm:SCSTImpliesWC} and Corollary \ref{cor:NonZeroIntImpliesWC} and G\"ottsche's assumption of the validity of the Kotschick--Morgan Conjecture \cite{KotschickMorgan} in his proof \cite{Goettsche} of the wall-crossing formula for Donaldson invariants. However, such a comparison overlooks the fact that our assumption of certain properties for local $\SO(3)$-monopole gluing maps is narrower and more specific. Indeed, our monograph \cite{FL5} effectively contains a proof of the Kotschick--Morgan Conjecture, modulo the assumption of certain technical properties for local gluing maps for anti-self-dual $\SO(3)$ connections which extend previous results of Taubes \cite{TauSelfDual, TauIndef, TauFrame}, Donaldson and Kronheimer \cite{DK}, and Morgan and Mrowka \cite{MorganMrowkaTube, MrowkaThesis}. Our proof of Theorem \ref{thm:Cobordism} in \cite{FL5} relies on our construction of a global gluing map for $\SO(3)$ monopoles and that in turn builds on properties of local gluing maps for $\SO(3)$ monopoles; the analogous comments apply to the proof of the Kotschick--Morgan Conjecture.

\subsection{Background}
When defining the Seiberg--Witten invariants in \cite{Witten}, Witten also gave a quantum field theory
argument yielding
the relation in Conjecture \ref{conj:WC}.
Soon after, Pidstrigatch and Tyurin \cite{PTLocal} introduced the moduli space
of $\SO(3)$ monopoles to give a possible path towards a mathematically rigorous proof of
this conjecture.  In \cite{FL5},
we used the moduli space of $\SO(3)$ monopoles to prove ---
through the assumption of certain properties of local $\SO(3)$ monopole gluing maps (see 
\cite[Sections 7.8 \& 7.9]{FL5} and \cite[Remark 3.3]{FL6})
--- the $\SO(3)$ monopole
cobordism formula (Theorem \ref{thm:Cobordism}).  This formula gives
a relation between the Donaldson and Seiberg--Witten invariants similar
to Witten's Conjecture \ref{conj:WC}, but contains a number of undetermined universal coefficients.
In \cite{FL2b,FLLevelOne} we computed some of these coefficients directly
while in \cite{FL6} we computed more by comparison with known examples.
Although these computations
showed that Theorem \ref{thm:Cobordism} implied Conjecture \ref{conj:WC} for a wide
range of standard four-manifolds, they did not suffice for all.
In this article, we use the methods of \cite{FL6} to show that the coefficients not determined in \cite[Proposition 4.8]{FL6}
are polynomials in one of the parameters on which they depend.
By combining this polynomial dependence with the vanishing condition in
the definition of superconformal simple type \eqref{eq:SWPolynomial}, we can show that
the sum over the terms in the cobordism formula containing these unknown coefficients vanishes.
Hence, the coefficients computed in \cite[Proposition 4.8]{FL6} suffice to determine the Donaldson
invariant in terms of Seiberg--Witten invariants and we show that the resulting expression satisfies
Conjecture \ref{conj:WC}.

Proofs of Conjecture \ref{conj:WC} for restricted
classes of standard four-manifolds have appeared elsewhere.
In \cite{FSRationalBlowDown}, Fintushel and Stern
proved Conjecture \ref{conj:WC} for elliptic surfaces and their blow-ups and rational
blow-downs.
Kronheimer and Mrowka in \cite[Corollary 7]{KMPropertyP} proved that
the cobordism formula in Theorem \ref{thm:Cobordism} implied Conjecture \ref{conj:WC} for standard
four-manifolds with
a tight surface with positive self-intersection, a sphere with self-intersection $(-1)$, and Euler number and signature equal
to that of a smooth hypersurface in $\CC\PP^3$ of even degree at least six.
In \cite{FL6}, we generalized the result of Kronheimer--Mrowka to standard four-manifolds
of Seiberg--Witten simple type satisfying $c(X)\le 3$ or which are
{\em abundant\/} in the sense that $B(X)^\perp\subset H^2(X;\ZZ)$, the orthogonal complement of the basic classes
with respect to the intersection form, contained a hyperbolic summand.
(We note that by \cite[Section A.2]{FL2a}, all simply-connected, closed, complex surfaces with $b^+\ge 3$
are abundant.)

T. Mochizuki \cite{Mochizuki_2009} proved a formula (see Theorem 4.1 in \cite{Goettsche_Nakajima_Yoshioka_2011}) expressing the Donaldson invariants of a complex projective surface
in a form similar to that given by the $\SO(3)$-monopole cobordism formula
(our Theorem \ref{thm:Cobordism}), but with coefficients
given as the residues of an explicit $\CC^*$-equivariant integral over the product of Hilbert schemes of points on $X$.
In \cite{Goettsche_Nakajima_Yoshioka_2011}, G\"ottsche, Nakajima, and Yoshioka
express a generating function for these integrals as a meromorphic one-form,
given by the \lq\lq leading terms \dots of Nekrasov's deformed partition function for the $N=2$ SUSY gauge theory with a single fundamental matter\rq\rq (\cite[p. 309]{Goettsche_Nakajima_Yoshioka_2011}).
By extending their meromorphic one-form to $\PP^1$ and analyzing the residues of this form
at its poles, the authors of \cite{Goettsche_Nakajima_Yoshioka_2011}
show that all four-manifolds whose Donaldson invariants are given by
Mochizuki's formula satisfy Witten's Conjecture.
This computation implies that the coefficients in Mochizuki's formula depend on the same data as the
coefficients in the $\SO(3)$-monopole cobordism formula (see \eqref{eq:Coefficients}) and
G\"ottsche, Nakajima, and Yoshioka conjecture (see
\cite[Conjecture 4.5]{Goettsche_Nakajima_Yoshioka_2011}) that
Mochizuki's formula (and thus their proof of Witten's Conjecture)
holds for all standard four-manifolds and not just complex projective surfaces.
It is worth noting that the superconformal simple type condition also appears in the proof in
\cite{Goettsche_Nakajima_Yoshioka_2011}, specifically \cite[Propositions 8.8 and 8.9]{Goettsche_Nakajima_Yoshioka_2011},
but as it is used to analyze the residue of the meromorphic form  at one of its poles,
superconformal simple type seems to play a role in \cite{Goettsche_Nakajima_Yoshioka_2011} which is different from that in our article.

The proof in \cite{FL6} that the $\SO(3)$ monopole cobordism formula implies Witten's Conjecture   used the result of \cite{FKLM} that abundant four-manifolds have
superconformal simple type.  In this article, we prove that Theorem \ref{thm:Cobordism} implies Conjecture \ref{conj:WC}
directly from the superconformal simple type condition.  The examples of non-abundant four-manifolds
given in \cite{FKLM} (following \cite{GompfMrowka},
one takes log transforms on tori in three disjoint nuclei of a K3 surface) show that there
are non-abundant four-manifolds which still satisfy the superconformal simple type condition.  Hence,
the results obtained here are strictly stronger than those in \cite{FKLM}.

In \cite{MMPdg,MMPhep},
Mari\~no, Moore, and Peradze originally defined
the concept of superconformal simple type in the context
of supersymmetric quantum field theory and, within that framework,
showed that a four-manifold satisfying the superconformal simple type
condition obeys the vanishing condition \eqref{eq:SWPolynomial}.  They conjectured
(see \cite[Conjecture 7.8.1]{MMPhep})
that all standard four-manifolds of Seiberg--Witten simple type obey \eqref{eq:SWPolynomial}.  Not only do all known examples of
standard four-manifolds satisfy \eqref{eq:SWPolynomial} (see \cite[Section 7]{MMPhep})
but the condition is preserved under the standard surgery operations
(blow-up, torus sum, and rational blow-down) used to construct new examples.
Using \eqref{eq:SWPolynomial} as a definition of superconformal simple type, they rigorously derived a lower bound on the number of basic classes for manifolds of superconformal simple type (see \cite[Theorem 8.1.1]{MMPhep}) in terms of topological invariants of the manifold.  Hence, the condition of
superconformal simple type is not only of interest to physicists but
has important mathematical implications as evidenced by
\cite[Theorem 8.1.1]{MMPhep},
\cite[Propositions 8.8]{Goettsche_Nakajima_Yoshioka_2011}, and
Theorem \ref{thm:SCSTImpliesWC}.

Finally, we note that the results of \cite{FL8} use  a
variant
of the $\SO(3)$-monopole cobordism formula
to prove that if $X$ is a standard four-manifold of Seiberg--Witten simple type,
then $X$ has superconformal simple type.  Combining this result with
Theorem \ref{thm:SCSTImpliesWC} gives Corollary \ref{cor:NonZeroIntImpliesWC}
which completes this part of the $\SO(3)$-monopole program.

\subsection{Outline}
After reviewing the definitions of the Seiberg--Witten and Donaldson invariants
and the superconformal simple type condition
in Section \ref{sec:Prelim}, we introduce the $\SO(3)$-monopole cobordism formula
and some useful reformulations of Conjecture \ref{conj:WC}
in Section \ref{sec:MonopoleCobordism}.
The technical heart of the paper appears in Section \ref{sec:CoeffComp}.
In Section \ref{subsec:AlgPrelim}, 
we cite an algebraic condition, Lemma \ref{lem:AlgCoeff}, that states when polynomial
equations determine coefficients
and review some basic results on difference operators in
Section \ref{subsec:DiffEquation}.
In Section \ref{subsec:ExManifoldsBlowUp},
we apply Lemma \ref{lem:AlgCoeff} to the blow-ups of some examples of
standard four-manifolds constructed in \cite{FSParkSympOneBasic} which
satisfy Conjecture \ref{conj:WC} to show that the coefficients appearing
in the $\SO(3)$-monopole cobordism formula are either determined, as in
Proposition \ref{prop:HighDegreeCoefficients}, or satisfy
a difference equation which determine them up to a polynomial,
as in Proposition  \ref{prop:LowICoefficientDifferenceRelation}.
Finally, in Section \ref{sec:MainProof} we prove the crucial
Lemma \ref{lem:SCSTVanishingSum}
which gives a polarized version of the vanishing condition on Seiberg--Witten polynomials
appearing in \eqref{eq:SWPolynomial}.
Combining Lemma \ref{lem:SCSTVanishingSum} with the polynomial dependence
of the unknown coefficients shows that the terms with these coefficients
can be ignored in the sum giving Donaldson's invariant, thus proving
Conjecture \ref{conj:WC}.


\subsection{Acknowledgements}
The authors would like to thank Ron Fintushel, Inanc Baykur and Nikolai Saveliev for helpful discussions on examples of four-manifolds as well as Tom Mrowka and Simon Donaldson for their support of this project. We also thank the anonymous referees for their comments and careful reading of our manuscript. Paul Feehan is grateful for support from the National Science Foundation under grant DMS-1510064 and Thomas Leness is grateful for support from the National Science Foundation grant DMS-1510063.

\section{Preliminaries}
\label{sec:Prelim}
We now review the definitions and basic properties of the relevant invariants.
\subsection{Seiberg--Witten invariants}
\label{subsec:SWDefn}
Detailed expositions of the theory of Seiberg--Witten invariants, introduced by Witten in \cite{Witten},
are provided in \cite{KMBook,MorganSWNotes,NicolaescuSWNotes}.
These invariants define an integer-valued map with finite support,
$$
\SW_X:\Spinc(X)\to\ZZ,
$$
on the set of \spinc structures on $X$.
A \spinc structure, $\fs=(W^\pm,\rho_W)$ on $X$, consists of a pair of complex rank-two bundles
$W^\pm\to X$ and a Clifford multiplication map $\rho:T^*X\to\Hom_\CC(W^+,W^-)$.
If $\fs\in\Spinc(X)$, then $c_1(\fs):=c_1(W^+)\in H^2(X;\ZZ)$
is characteristic.

One calls $c_1(\fs)$ a {\em Seiberg--Witten basic class\/} if $\SW_X(\fs)\neq 0$.
Define
\begin{equation}
\label{eq:SWBasic}
B(X)=
\{c_1(\fs): \SW_X(\fs)\neq 0\}.
\end{equation}
If $H^2(X;\ZZ)$ has 2-torsion, then $c_1:\Spinc(X)\to H^2(X;\ZZ)$ is not injective.
Because we will work with functions involving real homology and cohomology, we  define
\begin{equation}
\label{eq:DefineCohomSW}
\SW_X':H^2(X;\ZZ)\to\ZZ,
\quad
K\mapsto \sum_{\fs\in c_1^{-1}(K)}\SW_X(\fs).
\end{equation}
With the preceding definition, Witten's Formula \eqref{eq:WConjecture} is equivalent to
\begin{equation}
\label{eq:WConjCohom}
\bD^{w}_X(h)
=
2^{2-(\chi_h-c_1^2)}e^{Q_X(h)/2}
\sum_{K\in B(X)}(-1)^{\half(w^{2}+K\cdot w)}
SW'_X(K)e^{\langle K,h\rangle}.
\end{equation}
A four-manifold $X$ has {\em Seiberg--Witten simple type\/} if $\SW_X(\fs)\neq 0$ implies that $c_1^2(\fs)=c_1^2(X)$.

As discussed in \cite[Section 6.8]{MorganSWNotes},
there is an involution on $\Spinc(X)$, denoted by $\fs\mapsto\bar\fs$ and defined essentially by taking
the complex conjugate vector bundles, and having the property that
$c_1(\bar\fs)=-c_1(\fs)$.  By \cite[Corollary 6.8.4]{MorganSWNotes}, one has
$\SW_X(\bar\fs)=(-1)^{\chi_h(X)}\SW_X(\fs)$ and so
$B(X)$ is closed under the action of $\{\pm 1\}$ on $H^2(X;\ZZ)$.

Versions of the following result have appeared in
\cite{FSTurkish}, \cite[Theorem 14.1.1]{Froyshov_2008}, and \cite[Theorem 4.6.7]{NicolaescuSWNotes}.

\begin{thm}[Blow-up formula for Seiberg--Witten invariants]
\label{thm:FroyshovSWBlowUp}
\cite[Theorem 14.1.1]{Froyshov_2008}
Let $X$ be a standard four-manifold and let
$\widetilde X=X\#\bar{\CC\PP}^2$ be its blow-up.
Then $\widetilde X$ has Seiberg--Witten simple type if and only if that is true for $X$.
If $X$ has Seiberg--Witten simple type, then
\begin{equation}
\label{eq:SWBasicsOfBlowUp}
B(\widetilde X)=
\{K\pm e^*: \text{$K\in B(X)$}\},
\end{equation}
where $e^*\in H^2(\widetilde X;\ZZ)$ is the Poincar\'e dual of the exceptional curve, and if $K\in B(X)$, then $$
\SW_{\widetilde X}'(K\pm e^*)=\SW_X'(K).
$$
\end{thm}

\subsection{Donaldson invariants}
\label{subsec:DonDefn}
In \cite[Section 2]{KMStructure}, Kronheimer and Mrowka defined the Donaldson
series which encodes the Donaldson invariants developed in \cite{DonPoly}.
For  $w\in H^{2}(X;\ZZ)$, the
\emph{Donaldson invariant} is a linear function,
$$
D^{w}_{X}:\AAA(X) \to \RR,
$$
where $\AAA(X)$ is the symmetric algebra,
$$
\AAA(X) = \Sym(H_{\even}(X;\RR)).
$$
For $h\in H_2(X;\RR)$ and a generator $x\in H_0(X;\ZZ)$,
we define $D_X^w(h^{\delta-2m}x^m)=0$ unless
\begin{equation}
\label{eq:DegreeParity}
    \delta\equiv -w^{2}-3\chi_h(X)\pmod{4}.
  \end{equation}
When \eqref{eq:DegreeParity} is obeyed, then we adopt the definition of $D_X^w(h^{\delta-2m}x^m)$ given by Kronheimer and Mrowka in \cite[Section 2]{KMStructure}.
A four-manifold has {\em Kronheimer--Mrowka simple type\/} if for all $w\in H^2(X;\ZZ)$ and
all $z\in \AAA(X)$ one has
\begin{equation}
\label{eq:KMSimpleType}
D^{w}_{X}(x^{2}z)=4D^{w}_{X}(z).
\end{equation}
This equality implies that the Donaldson invariants are determined by
the  \emph{Donaldson series}, the formal power series
\begin{equation}
\label{eq:DefineDonaldsonSeries}
\bD^{w}_{X}(h) = D^{w}_{X}((1+\textstyle{\frac{1}{2}} x)e^{h}),
\quad h \in H_{2}(X;\RR).
\end{equation}
The following result allows us to work with a convenient choice of $w$:

\begin{prop}
\cite{KMStructure},
\cite[Theorem 2]{MunozBasicNonSimple}
\label{prop:IndepOfWCFromw}
Let $X$ be a standard four-manifold of Seiberg--Witten simple type. If Witten's Conjecture \ref{conj:WC}
holds for one $w\in H^2(X;\ZZ)$, then it holds for all $w\in H^2(X;\ZZ)$.
\end{prop}

The result below allows us to replace a manifold by its blow-up without loss of generality.

\begin{thm}
\label{thm:WCBlowDownInvariance}
\cite[Theorem 8.9]{FSRationalBlowDown}
Let $X$ be a standard four-manifold. Then Witten's Conjecture \ref{conj:WC} holds for $X$ if and only if it holds for the blow-up, $\widetilde X$.
\end{thm}

\subsection{Witten's conjecture}

It will be more convenient to have Witten's Conjecture \ref{conj:WC} expressed at the level of
the polynomial invariants rather than the power series they form.
Let $B'(X)$ be a fundamental domain for the action of $\{\pm 1\}$ on $B(X)$.

\begin{lem}
\label{lem:ReduceDFormToB'Sum}
\cite[Lemma 4.2]{FL6}
Let $X$ be a standard four-manifold.
Then $X$ satisfies equation \eqref{eq:WConjecture} and has Kronheimer--Mrowka simple type
if and only if the
Donaldson invariants of $X$ satisfy
$D^w_X(h^{\delta-2m}x^m)=0$ for $\delta\not\equiv -w^2-3\chi_h\pmod 4$
and for $\delta \equiv -w^2-3\chi_h\pmod 4$ satisfy
\begin{equation}
\begin{aligned}
\label{eq:DInvarForWCB'Sum}
{}&
D^w_X(h^{\delta-2m}x^m)
\\
{}&\quad
=
\sum_{\begin{subarray}{l}i+2k\\=\delta-2m\end{subarray}}
\sum_{K\in B'(X)}
(-1)^{{\eps(w,K)}}
\nu(K)
\frac{\SW'_X(K) (\delta-2m)!}{2^{k+c(X)-3-m} k!i!}
\langle K,h\rangle^i Q_X(h)^k,
\end{aligned}
\end{equation}
where
\begin{equation}
\label{eq:DefineOrientationEps}
\eps(w,K):=\frac{1}{2}(w^2+w\cdot K),
\end{equation}
and

\begin{equation}
\label{eq:DiracSpincFunction}
\nu(K)
=
\begin{cases}
    \frac{1}{2} & \text{if $K=0$,}
    \\
    1 &  \text{if $K\neq 0$.}
\end{cases}
\end{equation}
\end{lem}

\subsection{The superconformal simple type property}
A standard four-manifold $X$ has {\em superconformal simple type\/} if
$c(X)\le 3$ or $c(X)\ge 4$ and
for $w\in H^2(X;\ZZ)$ characteristic and all $h\in H_2(X;\RR)$
\begin{equation}
\label{eq:SCST}
\SW_X^{w,i}(h)
=\sum_{K\in B(X)} (-1)^{\eps(w,K)}\SW_X'(K)\langle K,h\rangle^i
=
0
\quad\text{for $i\le c(X)-4$}.
\end{equation}
Observe that we have rewritten \eqref{eq:SWPolynomial} as a sum
over $B(X)$ using the expression \eqref{eq:DefineCohomSW}. We further note that the property \eqref{eq:SCST} is invariant under blow-up.

\begin{lem}
\label{lem:SCSTBlowUp}
\cite[Theorem 7.3.1]{MMPhep},
\cite[Lemma 6.1]{FL8}
A standard manifold $X$ has superconformal simple type if and only if
its blow-up, $\widetilde X$, has superconformal simple type.
\end{lem}

\section{$\SO(3)$ monopoles and Witten's conjecture}
\label{sec:MonopoleCobordism}
The $\SO(3)$-monopole cobordism formula \eqref{eq:MainEquation} given in
Theorem \ref{thm:Cobordism}
provides an expression for the Donaldson invariant in
terms of the Seiberg--Witten invariants.

\begin{hyp}[Properties of local $\SO(3)$-monopole gluing maps]
\label{hyp:Local_gluing_map_properties}
The local gluing map, constructed in \cite{FL3}, gives a continuous parametrization of a neighborhood of $M_{\fs}\times\Si$ in $\bar\sM_{\ft}$ for each smooth stratum $\Si\subset\Sym^\ell(X)$.
\end{hyp}

Hypothesis \ref{hyp:Local_gluing_map_properties} is discussed in greater detail in 
\cite[Sections 7.8 \& 7.9]{FL5}. The question of how to assemble the \emph{local} gluing maps for neighborhoods of $M_{\fs}\times \Si$ in $\bar\sM_{\ft}$, as $\Si$ ranges over all smooth strata of $\Sym^\ell(X)$, into a \emph{global} gluing map for a neighborhood of $M_{\fs}\times \Sym^\ell(X)$ in $\bar\sM_{\ft}$ is itself difficult --- involving the so-called `overlap problem' described in \cite{FLMcMaster} --- but one which we do solve in \cite{FL5}. See Remark \ref{rmk:GluingThmProperties} for a further discussion of this point.

\begin{thm}[$\SO(3)$-monopole cobordism formula]
\cite{FL5}
\label{thm:Cobordism}
Let $X$ be a
standard four-manifold
of Seiberg--Witten simple type and assume Hypothesis \ref{hyp:Local_gluing_map_properties}.
Assume further that $w,\La\in H^2(X;\ZZ)$ and $\delta,m\in\NN$ satisfy
\begin{subequations}
\label{eq:CobordismConditions}
\begin{align}
\label{eq:CobordismCondition1}
w-\La\equiv w_2(X)\pmod 2,
\\
\label{eq:CobordismCondition2}
I(\La)=\La^2+c(X)+4\chi_h(X)>\delta,
\\
\label{eq:CobordismCondition3}
\delta\equiv -w^2-3\chi_h(X)\pmod 4,
\\
\label{eq:CobordismCondition4}
\delta-2m\ge 0.
\end{align}
\end{subequations}
Then, for any $h\in H_2(X;\RR)$ and positive generator $x\in H_0(X;\ZZ)$, we have
\begin{equation}
\label{eq:MainEquation}
\begin{aligned}
{}&
D^w_X(h^{\delta-2m}x^m)
\\
{}&
\quad=
\sum_{K\in B(X)}
(-1)^{\thalf(w^2-\si)+\thalf(w^2+(w-\La)\cdot K)}SW'_X(K)
f_{\delta,m}(\chi_h(X),c_1^2(X),K,\La)(h),
\end{aligned}
\end{equation}
where
the map,
$$
f_{\delta,m}(h):\ZZ\times\ZZ\times H^2(X;\ZZ)\times H^2(X;\ZZ) \to \RR[h],
$$
taking values in the ring of polynomials in the variable $h$ with
real coefficients, is universal (independent of $X$) and given by
\begin{equation}
\label{eq:Coefficients}
\begin{aligned}
{}&
f_{\delta,m}(\chi_h(X),c_1^2(X),K,\La)(h)
\\
{}&\quad:=
\sum_{\begin{subarray}{l}i+j+2k\\=\delta-2m\end{subarray}}
a_{i,j,k}(\chi_h(X),c_1^2(X),K\cdot\La,\La^2,m)
\langle K,h\rangle^i
\langle \La,h\rangle^j
Q_X(h)^k,
\end{aligned}
\end{equation}
and,
for each triple of non-negative integers, $i, j, k \in \NN$,
the coefficients,
$$
a_{i,j,k}:\ZZ\times\ZZ\times\ZZ\times\ZZ\times\NN \to \RR,
$$
are functions of the variables $\chi_h(X)$, $c_1^2(X)$, 
$K\cdot \La$, $\La^2$, and $m$ that are independent of $X$.
\end{thm}

\begin{rmk}
\label{rmk:CoefficientParityVanishing}
Because the coefficients $a_{i,j,k}(\chi_h(X),c_1^2(X),K\cdot\La,\La^2,m)$ only appear in a cobordism formula
of the type \eqref{eq:MainEquation} when $K\in H^2(X;\ZZ)$ is characteristic,
the coefficients $a_{i,j,k}(\chi_h(X),c_1^2(X),K\cdot\La,\La^2,m)$ are defined to be zero unless
$K\cdot\La\equiv \La^2\pmod 2$.
\end{rmk}


\begin{rmk}
\label{rmk:GluingThmProperties}
The proof of Theorem \ref{thm:Cobordism} in \cite{FL5} assumes the
Hypothesis \ref{hyp:Local_gluing_map_properties} (see 
\cite[Section 7.8]{FL5})
that the local gluing map for a neighborhood of $M_{\fs}\times \Si$ in $\bar\sM_{\ft}$ gives a continuous parametrization of a neighborhood of $M_{\fs}\times\Si$ in $\bar\sM_{\ft}$, for each smooth stratum $\Si\subset\Sym^\ell(X)$. These local gluing maps are the analogues for $\SO(3)$ monopoles of the local gluing maps for anti-self-dual $\SO(3)$ connections constructed by Taubes in \cite{TauSelfDual, TauIndef, TauFrame} and Donaldson and Kronheimer in \cite[\S 7.2]{DK}; see also \cite{MorganMrowkaTube, MrowkaThesis}. We have established the existence of local gluing maps in \cite{FL3} and expect that a proof of the continuity for the local gluing maps with respect to Uhlenbeck limits should be similar to our proof in \cite{FLKM1} of this property for the local gluing maps for anti-self-dual $\SO(3)$ connections. The remaining properties of local gluing maps assumed in \cite{FL5} are that they are injective and also surjective in the sense that elements of $\bar\sM_{\ft}$ sufficiently close (in the Uhlenbeck topology) to $M_{\fs}\times\Si$ are in the image of at least one of the  local gluing maps. In special cases, proofs of these properties for the local gluing maps for anti-self-dual $\SO(3)$ connections (namely, continuity with respect to Uhlenbeck limits, injectivity, and surjectivity) have been given in \cite[\S 7.2.5, 7.2.6]{DK}, \cite{TauSelfDual, TauIndef, TauFrame}. The authors are currently developing a proof of the required properties for the local gluing maps for $\SO(3)$ monopoles \cite{Feehan_Leness_monopolegluingbook}.
Our proof will also yield the analogous properties for the local gluing maps for anti-self-dual $\SO(3)$ connections, as required to complete the proof of the Kotschick--Morgan Conjecture \cite{KotschickMorgan}, based on our work in \cite{FL5}.
\end{rmk}

It will be convenient for us to rewrite Theorem \ref{thm:Cobordism}
as a sum over $B'(X)\subset B(X)$, a fundamental domain for the action of $\{\pm 1\}$,
to compare with Lemma \ref{lem:ReduceDFormToB'Sum}. To this end,
we follow \cite[Equation (4.4)]{FL6} and define
\begin{multline*}
b_{i,j,k}(\chi_h(X),c_1^2(X),K\cdot\La,\La^2,m)
\\
:=
(-1)^{c(X)+i}a_{i,j,k}(\chi_h(X),c_1^2(X),-K\cdot\La,\La^2,m)
\\
+ a_{i,j,k}(\chi_h(X),c_1^2(X),K\cdot\La,\La^2,m),
\end{multline*}
where $a_{i,j,k}$ are the coefficients appearing
in \eqref{eq:Coefficients}.
To simplify the orientation factor in \eqref{eq:MainEquation}, we
define
\begin{multline}
\label{eq:OrientedCoeff}
\tilde b_{i,j,k}(\chi_h(X),c_1^2(X),K\cdot\La,\La^2,m)
\\
:=
(-1)^{\thalf(\La^2+\La\cdot K)}
b_{i,j,k}(\chi_h(X),c_1^2(X),K\cdot\La,\La^2,m).
\end{multline}
Observe that
\begin{multline}
\label{eq:SignChangeOfTildeBCoeff}
\tilde b_{i,j,k}(\chi_h(X),c_1^2(X),-K\cdot\La,\La^2,m)
\\
=
(-1)^{c(X)+i+\La\cdot K}
\tilde b_{i,j,k}(\chi_h(X),c_1^2(X),K\cdot\La,\La^2,m).
\end{multline}
We now rewrite \eqref{eq:MainEquation} as a sum over $B'(X)$.

\begin{lem}
\label{lem:ReduceCobordismFormToB'Sum}
Assume the hypotheses of Theorem \ref{thm:Cobordism}.
Denote the coefficients in \eqref{eq:SignChangeOfTildeBCoeff} more concisely by
$$
\tilde b_{i,j,k}(K\cdot\La) := \tilde b_{i,j,k}(\chi_h(X),c_1^2(X),K\cdot\La,\La^2,m).
$$
Then, for $\eps(w,K)=\thalf(w^2+w\cdot K)$ as in \eqref{eq:DefineOrientationEps},
\begin{equation}
\begin{aligned}
\label{eq:CompareCoeff2}
D^w_X(h^{\delta-2m}x^m)
&=
\sum_{K\in B'(X)}
\sum_{\begin{subarray}{l}i+j+2k\\=\delta-2m\end{subarray}}
\nu(K)(-1)^{{\eps(w,K)}}SW'_X(K)
\\
&\qquad\times
\tilde b_{i,j,k}(K\cdot\La)
\langle K,h\rangle^i
\langle \La,h\rangle^j
Q_X(h)^k,
\end{aligned}
\end{equation}
where $\nu(K)$ is defined by \eqref{eq:DiracSpincFunction}.
\end{lem}

\begin{proof}
We first compare the orientation factors of $\eps(w,K)$ appearing in \eqref{eq:DInvarForWCB'Sum}
and $\thalf(w^2-\si)+\thalf(w^2+(w-\La)\cdot K)$ appearing in
\eqref{eq:MainEquation}.  Because $w-\La$ is characteristic by
\eqref{eq:CobordismCondition1}, we have
\begin{subequations}
\begin{align}
\label{eq:CharSquaredIsSignature}
\si(X)&\equiv (w-\La)^2 \pmod 8\quad\text{(by \cite[Lemma 1.2.20]{GompfStipsicz})},
\\
\label{eq:CharMultEqualsSquare}
\La\cdot(w-\La)&\equiv \La^2\pmod 2.
\end{align}
\end{subequations}
Then,
\begin{align*}
{}&\frac{1}{2}(w^2-\si(X))+\frac{1}{2}(w^2+(w-\La)\cdot K)
\\
{}&\quad\equiv
\eps(w,K)
+\frac{1}{2}\left(  w^2 -\si(X)\right)
-\frac{1}{2}\La\cdot K
\pmod 2
\quad\hbox{(by \eqref{eq:DefineOrientationEps})}
\\
{}&\quad\equiv
\eps(w,K)
+
\frac{1}{2}(w^2 -(w-\La)^2
-\La\cdot K) \pmod 2
\quad\text{(by \eqref{eq:CharSquaredIsSignature})}
\\
{}&\quad\equiv
\eps( w,K)
+
\frac{1}{2}( 2w\cdot\La-\La^2 -\La\cdot K) \pmod 2
\\
{}&\quad\equiv
\eps( w,K)
+
\frac{1}{2}( 2 w\cdot\La-2\La^2+\La^2 -\La\cdot K) \pmod 2
\\
{}&\quad\equiv
\eps( w,K)
-
(\La- w)\cdot\La +\frac{1}{2}(\La^2-\La\cdot K) \pmod 2
\\
{}&\quad\equiv
\eps(w,K)
-
\La^2 +\frac{1}{2}(\La^2-\La\cdot K) \pmod 2
\quad\text{(by \eqref{eq:CharMultEqualsSquare})},
\end{align*}
and hence,
\begin{equation}
\label{eq:OrientComp}
\frac{1}{2}(w^2-\si(X))+\frac{1}{2}(w^2+(w-\La)\cdot K)
\equiv
\eps(w,K)
 -\frac{1}{2}(\La^2+\La\cdot K) \pmod 2.
\end{equation}
From \cite[Lemma 4.3]{FL6},we have
\begin{equation}
\begin{aligned}
\label{eq:FL6_lemma_4-7}
D^w_X(h^{\delta-2m}x^m)
&=
\sum_{K\in B'(X)}
\sum_{\begin{subarray}{l}i+j+2k\\=\delta-2m\end{subarray}}
\nu(K)(-1)^{\thalf(w^2-\si(X))+\thalf(w^2+(w-\La)\cdot K)}SW'_X(K)
\\
&\qquad\times
b_{i,j,k}(\chi_h(X),c_1^2(X),K\cdot\La,\La^2,m)
\langle K,h\rangle^i
\langle \La,h\rangle^j
Q_X(h)^k,
\end{aligned}
\end{equation}
The result \eqref{eq:CompareCoeff2} now follows from \eqref{eq:OrientComp},
\eqref{eq:FL6_lemma_4-7},
and the relation between
the coefficients $\tilde b_{i,j,k}$ and $b_{i,j,k}$ in  \eqref{eq:OrientedCoeff}.
\end{proof}

The following lemma allows us to ignore the coefficients $\tilde b_{0,j,k}$
for the purpose of proving Theorem \ref{thm:SCSTImpliesWC} and Corollary \ref{cor:NonZeroIntImpliesWC}.

\begin{lem}
\label{lem:BlownUpCobordism}
Continue the notation and hypotheses of Lemma \ref{lem:ReduceCobordismFormToB'Sum}.
Then,
\begin{multline}
\label{eq:CobordismBlowUp}
D^w_X(h^{\delta-2m}x^m)
=
\sum_{K\in B'(X)}
\sum_{\begin{subarray}{l}i+j+2k\\=\delta-2m\end{subarray}}
(-1)^{\eps(w,K)}\SW'_X(K)\frac{2(i+1)}{(\delta-2m+1)}
\\
\qquad\times\tilde b_{i+1,j,k}(\chi_h(X),c_1^2(X)-1,K\cdot\La,\La^2,m)
\\
\times\langle K,h\rangle^i\langle \La,h\rangle^j Q_X(h)^k.
\end{multline}
\end{lem}

\begin{proof}
Let $\widetilde X\to X$ be the blow-up of $X$ at one point, let $e\in H_2(\widetilde X;\ZZ)$
be the fundamental class of the exceptional curve,
and let $e^*\in H^2(\widetilde X;\ZZ)$ be the Poincar\'e dual of $e$.
Using the direct sum decomposition of the homology and cohomology of $\widetilde X$,
we will consider both the homology and cohomology of $X$ as subspaces of those of $\widetilde X$.
Denote $\tilde w:=w+e^*$.  The blow-up formula
\cite{KotschickBPlus1, LenessBlowUp} gives
\begin{equation}
\label{eq:DonaldsonBlowup}
D^w_X(h^{\delta-2m}x^m)
=
D^{\tilde w}_{\widetilde X}(h^{\delta-2m}ex^m).
\end{equation}
By Theorem \ref{thm:FroyshovSWBlowUp},
\begin{equation}
\label{eq:BasicClasssesOfBlowUp}
B'(\widetilde X)=\{K_\varphi=K+(-1)^\varphi e^*: K\in B'(X),\ \varphi\in\ZZ/2\ZZ\}.
\end{equation}
To apply the cobordism formula \eqref{eq:CompareCoeff2} to compute the
right-hand-side of \eqref{eq:DonaldsonBlowup}, we must discuss the isomorphism
$\Phi$ from the space of symmetric, $d$-linear functionals
on a real vector space $V$ onto the space of degree-$d$ polynomials on $V$, defined by
(see \cite[Section 6.1.1]{FrM})
$$
\Phi(M)(h) := M(\underbrace{h,\dots,h}_\text{$d$ copies}).
$$
If $F=\Phi(M)$ is a degree $d$-polynomial, we can find $M$ by the formula \cite[p. 396]{FrM},
\begin{equation}
\label{eq:Polarization}
M(h_1,\dots,h_d)
=
\left.\frac{1}{d!}
\frac{\rd^d}{\rd t_1\rd t_2\dots\rd t_d}
F(t_1h_1+\dots + t_dh_d)\right|_{t_1=\dots=t_d=0}.
\end{equation}
For the polynomial of degree $\delta-2m+1=i+j+2k$ defined by
$$
F^\varphi_{i,j,k}(\tilde h)
:=
\langle K_\varphi,\tilde h\rangle^i\langle \La,\tilde h\rangle^j Q_{\widetilde X}(\tilde h)^k,
\quad\forall\, \tilde h\in H_2(\widetilde X;\RR),
$$
where (as usual) $\La\in H^2(X;\ZZ)$, the identity \eqref{eq:Polarization} implies that the functional $M^\varphi_{i,j,k}:=\Phi^{-1}(F^\varphi_{i,j,k})$ satisfies
\begin{equation}
\label{eq:BlowUpPolarization}
M^\varphi_{i,j,k}(e,h,\dots,h)
=
\frac{i(-1)^{\varphi+1}}{(\delta-2m+1)}
\langle K,h\rangle^{i-1}\langle \La,h\rangle^j Q_{X}(h)^k,
\quad\forall\, h\in H_2(X;\RR).
\end{equation}
Before applying the cobordism formula
to the right-hand-side of \eqref{eq:DonaldsonBlowup}, we check that the
conditions \eqref{eq:CobordismConditions} of Theorem \ref{thm:Cobordism} hold.

Our assumption that $\tilde w=w+e^*$ ensures that
for $\La\in H^2(X;\ZZ)$, we have $\La+\tilde w\equiv w_2(\widetilde X)$ if and only if
$\La+w\equiv w_2(X)$.  Hence, the condition
\eqref{eq:CobordismCondition1} holds for $w$ and $\La$ on $X$
if and only if it holds for $\tilde w$ and $\La$ on $\widetilde X$.

Because $c(\widetilde X)=c(X)+1$ and $\chi_h(\widetilde X)=\chi_h(X)$, we see that
$\La^2+c(X)+4\chi_h(X)>\delta$ if and only if
$\La^2+c(\widetilde X)+4\chi_h(\widetilde X)>\delta+1$.
Consequently, the condition \eqref{eq:CobordismCondition2} holds for $\La$, $\delta$ and
$X$ if and only if it holds for $\La$, $\delta+1$, and $\widetilde X$.

Since $-\tilde w^2=-w^2+1$, we have $\delta+1\equiv \tilde w^2-3\chi_h(\widetilde X)\pmod 4$
if and only if $\delta\equiv \tilde w^2-3\chi_h(X)\pmod 4$.
Therefore, the condition \eqref{eq:CobordismCondition3} holds for
$w$, $\delta$ and $X$ if and only if it holds for $\tilde w$, $\delta+1$, and $\widetilde X$.

Finally, if the condition \eqref{eq:CobordismCondition4} holds for $\delta$ and $m$,
then it holds for $\delta+1$ and $m$.

Thus, if we assume that the conditions \eqref{eq:CobordismConditions}
in Theorem \ref{thm:Cobordism}
hold for $w$, $\La$, $\delta$, and $m$ on $X$, then they hold for
$\tilde w$, $\La$, $\delta+1$, and $m$ on $\widetilde X$.
Hence, we can apply \eqref{eq:CompareCoeff2}
in Lemma \ref{lem:ReduceCobordismFormToB'Sum}
to compute the right-hand-side of \eqref{eq:DonaldsonBlowup}.
Note that for any $K\in B'(X)$,
$$
\eps(w,K)
\equiv
\eps(\tilde w,K_1)
\equiv
\eps(\tilde w,K_0)+1
\pmod 2,
$$
where $K_0,K_1\in B'(\widetilde X)$ as in \eqref{eq:BasicClasssesOfBlowUp}.
Also, because $\La\in H^2(X;\ZZ)$,
$$
\tilde b_{i,j,k}(\chi_h(\widetilde X),c_1^2(\widetilde X),K_\varphi\cdot \La, \La^2,m)
=
\tilde b_{i,j,k}(\chi_h(X),c_1^2(X)-1,K\cdot\La,\La^2,m).
$$
Applying the definition of $M^\varphi_{i,j,k}$ from \eqref{eq:BlowUpPolarization} to
the cobordism formula \eqref{eq:CompareCoeff2}
for $\La$, $\tilde w$, and $\delta+1$ on $\widetilde X$
then gives us
\begin{align*}
{}&D^{\tilde w}_{\widetilde X}(h^{\delta-2m}ex^m)
\\
&=
\sum_{K_\varphi\in B'(\widetilde X)}
\sum_{\begin{subarray}{c}i+j+2k\\=\delta-2m+1\end{subarray}}
(-1)^{{\eps(w,K_\varphi)}}SW'_{\widetilde X}(K_\varphi)
\\
&\qquad\times
\tilde b_{i,j,k}(\chi_h(X),c_1^2(X)-1,K\cdot\La,\La^2,m)
M^\varphi_{i,j,k}(e,h\dots,h)
\\
{}&=
\sum_{K\in B'(X)}
\sum_{\begin{subarray}{c}i+j+2k\\=\delta-2m+1\end{subarray}}
(-1)^{\eps(w,K)}\SW_X(K)\tilde b_{i,j,k}(\chi_h(X),c_1^2(X)-1,K\cdot\La,\La^2,m)
\\
{}&\qquad\times
\left( M^1_{i,j,k}(e,h,\dots,h) - M^0_{i,j,k}(e,h\dots,h)\right)
\\
{}&=
\sum_{K\in B'(X)}
\sum_{\begin{subarray}{c}i+j+2k\\=\delta-2m+1\end{subarray}}
(-1)^{\eps(w,K)}\SW_X(K)\tilde b_{i,j,k}(\chi_h(X),c_1^2(X)-1,K\cdot\La,\La^2,m)
\\
{}&\qquad\times
\frac{2i}{(\delta-2m+1)}\langle K,h\rangle^{i-1}\langle \La,h\rangle^j Q_X(h)^k
\quad\text{(by \eqref{eq:BlowUpPolarization})}
\\
{}&=
\sum_{K\in B'(X)}
\sum_{\begin{subarray}{l}i+j+2k\\=\delta-2m\end{subarray}}
(-1)^{\eps(w,K)}\SW_X(K)\tilde b_{i+1,j,k}(\chi_h(X),c_1^2(X)-1,K\cdot\La,\La^2,m)
\\
{}&\qquad\times
\frac{2(i+1)}{(\delta-2m+1)}\langle K,h\rangle^i\langle \La,h\rangle^j Q_X(h)^k.
\end{align*}
Combining the preceding equalities  and \eqref{eq:DonaldsonBlowup} gives the result.
\end{proof}

\section{Constraining the coefficients}
\label{sec:CoeffComp}
In this section, we show that the coefficients $\tilde b_{i,j,k}$
appearing in \eqref{eq:CompareCoeff2} which are not determined by
\cite[Proposition 4.8]{FL6} satisfy a difference equation in
the parameter $K\cdot \La$ and thus can be written as a polynomial
in this parameter.

\subsection{Algebraic preliminaries}
\label{subsec:AlgPrelim}

To determine the coefficients $\tilde b_{i,j,k}$ appearing in \eqref{eq:CompareCoeff2},
we compare equations  \eqref{eq:DInvarForWCB'Sum} and \eqref{eq:CompareCoeff2}
on manifolds where Witten's Conjecture \ref{conj:WC} is known to hold and use
the following generalization of \cite[Lemma VI.2.4]{FrM}.

\begin{lem}
\label{lem:AlgCoeff}
\cite[Lemma 4.1]{FL6}
Let $V$ be a finite-dimensional real vector space.
Let $T_1,\dots,T_n $ be linearly independent elements of the dual
space $V^*$.
Let $Q$ be a quadratic form on $V$ which is non-zero
on $\cap_{i=1}^n\Ker T_i$.  Then $T_1,\dots,T_n,Q$ are algebraically
independent in the sense that if
$F(z_0,\dots,z_n)\in \RR[z_0,\dots,z_n]$
and $F(Q,T_1,\dots,T_n):V\to\RR$ is the zero map, then $F(z_0,\dots,z_n)$
is the zero element of $\RR[z_0,\dots,z_n]$.
\end{lem}

\subsection{Difference equations}
\label{subsec:DiffEquation}

We review some notation and results for difference
operators.  For $f:\ZZ\to\RR$ and $p,q\in\ZZ$, define
$$
(\nabla^q_pf)(x):=f(x)+(-1)^q f(x+p).
$$
For $a\in\ZZ/2\ZZ$ and $p\in \ZZ$, define $pa,ap\in\ZZ$ by
\begin{equation}
\label{eq:MultZ2andZ}
pa
=
ap
=-\frac{1}{2}\left(-1 +(-1)^a\right)p
=
\begin{cases}
0 & \text{if $a\equiv 0\pmod 2$},\\
p & \text{if $a\equiv 1\pmod 2$.}
\end{cases}
\end{equation}
We recall the

\begin{lem}
\cite[Lemma 4.6]{FL6}
\label{lem:PermutationSumAsDifferenceOperator}
For all $(p_1,\dots,p_n)$ and $(q_1,\dots,q_n)\in\ZZ^n$, there holds
$$
\sum_{\varphi\in(\ZZ/2\ZZ)^n}
(-1)^{\sum_{u=1}^n q_u\pi_u(\varphi)} f\left(x+\sum_{u=1}^n p_u \pi_u(\varphi)\right)
=
(\nabla^{q_1}_{p_1}\nabla^{q_2}_{p_2}\dots\nabla^{q_n}_{p_n}f)(x),
$$
where $\pi_u:(\ZZ/2\ZZ)^n\to\ZZ/2\ZZ$ is projection onto the $u$-th factor
and, for a constant
function $C$, there holds
\begin{equation}
\label{eq:ConstantDifference}
(\nabla^{q_n}_{p_n}\nabla^{q_{n-1}}_{p_{n-1}}\dots\nabla^{q_1}_{p_1}C)
=
\begin{cases}
0, & \text{if $\exists u$ with $1\le u\le n$ and $q_u\equiv 1\pmod 2$,}
\\
2^{n}C, & \text{if $q_u\equiv 0\pmod 2$ $\forall\, u$ with $1\le u\le n$.}
\end{cases}
\end{equation}
\end{lem}

We will also use the following similar result
(compare \cite[Lemma 2.22]{Elaydi}).

\begin{lem}
\label{lem:IteratedDifferenceOp}
For 
$f:\ZZ\to\RR$ and $\la\in\ZZ$, there holds
$$
\left(\left(\nabla^1_\la\right)^n f\right)(x)
=
\sum_{i=0}^n (-1)^i \binom{n}{i} f(x+i\la).
$$
\end{lem}

\begin{proof}
If $E_\la$ is the translation operator, $E_\la f(x)=f(x+\la)$, and $I$ is the identity, then $\nabla^1_\la =I-E_\la$.
The conclusion then follows from a binomial expansion.
\end{proof}

We have the following

\begin{lem}
\label{lem:OneDifference}
Let $\la\in \ZZ$ and $p:\ZZ\to\RR$ be a function.
\begin{enumerate}
\item If $\nabla^1_\la p(x)$ is a polynomial of degree $n$ in $x$, then $p(\la x)$ is a polynomial of degree $n+1$;

\item If $\nabla^1_\la p(x)=0$, then $p(\la x)$ is constant.
\end{enumerate}
\end{lem}

\begin{proof}
Note that the lemma is trivial if $\la=0$. The second statement follows easily from the definitions.

We prove the first statement by induction on $n$.
If $n=0$, then there is a constant $C_1$ such that $p(x)-p(x+\la)=C_1$ for all
$x$ and hence $p(\la x)=-C_1x+C_2$, where $C_2=p(0)$.

For the inductive step, assume that $\nabla^1_\la p(x)$ is a polynomial of degree $m$
and define $q(x):=p(\la x)$.  Because
$(\nabla^1_1 q)(x)=(\nabla^1_\la p)(\la x)$, we see that
$(\nabla^1_1 q)(x)= C x^m + r(x)$, where $r(x)$ is a
polynomial of degree $m-1$.  We compute that
$$
\nabla^1_1 \left( q(x) +\frac{C}{m+1} x^{m+1}\right)
$$
is a polynomial of degree $m-1$ and so, by induction, $q(x)+ C x^{m+1}/(m+1)$
is a polynomial of degree $m$.  Hence, $q(x)=p(\la x)$ is a
polynomial of degree $m+1$, completing the induction.
\end{proof}

\begin{cor}
\label{cor:DegreeNPolyn}
Let $n\geq 1$ be an integer and for $\la\neq 0$,
let $c:\ZZ\to\RR$ be a function satisfying,
$$
\underbrace{(\nabla^1_\la\nabla^1_{\la}\cdots \nabla^1_\la c)}_{\text{$n$ copies}}(\la x)=0,
$$
for all $x\in\ZZ$.  Then $c_{\la}(x)=c(\la x)$ is a polynomial in $x$ of degree $n-1$.
\end{cor}

\begin{proof}
From Lemma \ref{lem:IteratedDifferenceOp}, one can see that
$c_\la$ satisfies $(\nabla^1_1\cdots\nabla^1_1 c_\la)(x)=0$.
The result then follows from Lemma \ref{lem:OneDifference} and induction on $n$.
\end{proof}

\subsection{The example four-manifolds and blow-up formulas}
\label{subsec:ExManifoldsBlowUp}
In \cite[Section 4.2]{FL6}, we used the manifolds constructed by
Fintushel, Park and Stern  in \cite{FSParkSympOneBasic}
to give a family of standard
four-manifolds $X_q$, for $q=2,3,\dots$, obeying the following conditions:
\begin{enumerate}
\item
$X_q$ satisfies Witten's Conjecture \ref{conj:WC};
\item
For $q=2,3,\dots$, one has $\chi_h(X_q)=q$ and $c(X_q)=3$;
\item
$B'(X_q)=\{K\}$ with $K\neq 0$;
\item
For each $q$,
there are classes $f_1,f_2\in H^2(X_q;\ZZ)$ satisfying
\begin{subequations}
\begin{align}
\label{eq:HyperbolicSummand1}
{}& f_1\cdot f_2=1 \quad\text{and}\quad  f_i^2=0 \quad\hbox{and}\quad f_i\cdot K=0 \quad\hbox{for}\quad i=1,2,
\\
\label{eq:HyperbolicSummand2}
{}&\text{The cohomology classes $\{f_1,f_2,K\}$ are linearly independent in $H^2(X_q;\RR)$,}
\\
\label{eq:HyperbolicSummand3}
{}&\text{The restriction of $Q_{X_q}$ to $\Ker f_1\cap\Ker f_2\cap \Ker K$ is non-zero.}
\end{align}
\end{subequations}
\end{enumerate}
Let $X_q(n)$ be the blow-up of $X_q$ at $n$ points,
\begin{equation}
\label{eq:BlownUpExampleNotation}
X_q(n):=
X_q\underbrace{\# \overline{\CC\PP}^2\cdots \# \overline{\CC\PP}^2}_{\text{$n$ copies}}.
\end{equation}
Then $X_q(n)$ is a standard four-manifold of Seiberg--Witten simple type
and satisfies Witten's Conjecture \ref{conj:WC}
by Theorem \ref{thm:WCBlowDownInvariance},
with
\begin{equation}
\label{eq:BlownUpExampleChar}
\chi_h(X_q(n))=q,
\quad c_1^2(X_q(n))=q-n-3,
\quad\hbox{and}\quad c(X_q(n))=n+3.
\end{equation}
We will consider both the homology and cohomology of $X_q$ as subspaces
of those of $X_q(n)$.
Let $e_u^*\in H^2(X_q(n);\ZZ)$ be the Poincar\'e dual of the $u$-th exceptional class.
Let $\pi_u:(\ZZ/2\ZZ)^n\to\ZZ/2\ZZ$ be projection onto the $u$-th factor.
For $\varphi\in (\ZZ/2\ZZ)^n$, we define
\begin{equation}
\label{eq:BasicClassesOnBlowUp}
K_\varphi:=K+\sum_{u=1}^n (-1)^{\pi_u(\varphi)}e_u^*
\quad\hbox{and}\quad
K_0:=K+\sum_{u=1}^n e_u^*.
\end{equation}
By Theorem \ref{thm:FroyshovSWBlowUp},
\begin{equation}
\label{eq:BlowUpBasicClasses}
B'(X_q(n))
=
\{ K_\varphi:\varphi\in (\ZZ/2\ZZ)^n\},
\end{equation}
and, for all $\varphi\in (\ZZ/2\ZZ)^n$,
\begin{equation}
\label{eq:SWInvarOfBlowUpExamples}
\SW'_{X_q(n)}(K_\varphi)=\SW'_{X_q}(K).
\end{equation}
Because $X_q(n)$ has Seiberg--Witten simple type, we have
\begin{equation}
\label{eq:SWSimpleTypeOnBlowUp}
K_\varphi^2=c_1^2(X_q(n))\quad\text{for all $\varphi\in (\ZZ/2\ZZ)^n$}.
\end{equation}
In addition, because $K\neq 0$, we see that
\begin{equation}
\label{eq:NoZeroInB}
0\notin B'(X_q(n)).
\end{equation}
Noting that the manifolds $X_q(n)$ satisfy Witten's Conjecture \ref{conj:WC},
Lemma \ref{lem:AlgCoeff} and the equality given by combining
equations  \eqref{eq:DInvarForWCB'Sum} and \eqref{eq:CompareCoeff2}, applied to the manifolds $X_q(n)$,
will show that the coefficients $\tilde b_{i,j,k}$ satisfy certain difference equations. Those
difference equations will allow us to prove Theorem \ref{thm:SCSTImpliesWC}.

For $n\ge 2$, the set $B'(X_q(n))$ is not linearly independent
in $H^2(X_q(n);\RR)$.
To apply Lemma \ref{lem:AlgCoeff}, we need to replace $B'(X_q(n))$ with
a linearly independent set.  To this end, we give
the following formula for the Donaldson invariants of
$X_q(n)$.  It differs from \cite[Lemma 4.7]{FL6}
in the change of coefficients from $b_{i,j,k}$ to $\tilde b_{i,j,k}$
and in our use of the linearly independent set
$K\pm e_1^*,e_2^*,\dots,e_n^*$.

\begin{lem}
\label{lem:BlowUpCobordism}
For $n,q\in\ZZ$ with $n\ge 1$ and $q\ge 2$,
let $X_q(n)$ be the manifold defined in \eqref{eq:BlownUpExampleNotation}.
For $\La,w\in H^2(X_q;\ZZ)$ and $\delta,m\in\NN$ satisfying $\La-w\equiv w_2(X_q)\pmod 2$
and $\delta-2m\ge 0$, define $\tilde w,\tilde\La \in H^2(X_q(n);\ZZ)$ by
\begin{equation}
\label{eq:DefinetwtLa}
\tilde w:=w+\sum_{u=1}^n w_ue_u^*
\quad\hbox{and}\quad
\tilde\La:=\La+\sum_{u=1}^n \la_u e_u^*,
\end{equation}
where $w_u,\la_u\in\ZZ$ and $w_u+\la_u\equiv 1\pmod 2$ for $u=1,\dots,n$.
We assume that
\begin{subequations}
\begin{align}
\label{eq:BlowUpFormulaCondition1}
\La^2>\delta-(n+3)-4q +\sum_{u=1}^n \la_u^2,
\\
\label{eq:BlowUpFormulaCondition2}
\delta\equiv -w^2+\sum_{u=1}^n w_u^2-3q\pmod 4.
\end{align}
\end{subequations}
Denote $x:=\tilde K_\varphi\cdot\tilde\La$ and, for
$i,j,k\in\NN$ satisfying $i+j+2k+2m=\delta$, write
$$
\tilde b_{i,j,k}(x)
=
\tilde b_{i,j,k}(\chi_h(X_q(n)),c_1^2(X_q(n)),x,\tilde\La^2,m).
$$
Then, for $x_0=K_0\cdot\tilde\La$, where $K_0$ is defined in \eqref{eq:BasicClassesOnBlowUp}, the expressions
\begin{equation}
\label{eq:BlownUpUsefulCobordismFormula}
\begin{aligned}
{}&
\sum_{\begin{subarray}{c}i_1+\cdots+i_n+2k\\=\delta-2m\end{subarray}}
\frac{(\delta-2m)!}{2^{k+n-m}k!i_1!\cdots i_n!}
p^{\tilde w}(i_2,\dots,i_n)
\left(\prod_{u=2}^n \langle e_u^*,h\rangle^{i_u}\right)
Q_{X_q(n)}(h)^k
\\
{}&\qquad\times
\left(
\langle K+e_1^*,h\rangle^{i_1}
+
(-1)^{w_1}\langle K-e_1^*,h\rangle^{i_1}
\right)
\\
{}&\quad=
\sum_{\begin{subarray}{c}i_1+\cdots+i_n+j+2k\\=\delta-2m\end{subarray}}
\binom{i_1+\cdots+i_n}{i_1,\dots,i_n}
\langle \tilde\La,h\rangle^{j}
\left(\prod_{u=2}^n \langle e_u^*,h\rangle^{i_u}\right)
Q_{X_q(n)}(h)^k
\\
{}&\qquad\qquad
\times
\bigg(
\nabla^{i_2+w_2}_{2\la_2}\cdots \nabla^{i_n+w_n}_{2\la_n}
\tilde b_{i,j,k}(x_0)
\langle K+e_1^*,h\rangle^{i_1}
\\
{}&\qquad\qquad\quad
+
(-1)^{w_1}
\nabla^{i_2+w_2}_{2\la_2}\cdots \nabla^{i_n+w_n}_{2\la_n}
\tilde b_{i,j,k}(x_0+2\la_1)
\langle K-e_1^*,h\rangle^{i_1}
\bigg),
\end{aligned}
\end{equation}
are both equal to the following multiple of the Donaldson invariant,
\[
\frac{(-1)^{\eps(\tilde w,\varphi_0)}}{\SW'_{X_q}(K)}
D^{\tilde w}_{X_q(n)}(h^{\delta-2m}x^m),
\]
where $\tilde \La$ is as defined in \eqref{eq:DefinetwtLa} and
\begin{equation}
\label{eq:DefineSumFactor1}
p^{\tilde w}(i_2,\dots,i_n)
=
\begin{cases}
0 & \text{if $\exists u$ with $2\le u\le n$ and $w_u+i_u\equiv 1\pmod 2$,}
\\
2^{n-1} & \text{if $w_u+i_u\equiv 0\pmod 2$ $\forall\, u$ with $2\le u\le n$.}
\end{cases}
\end{equation}
\end{lem}

\begin{proof}
We first verify that $\tilde\La$, $\tilde w$, $\delta$, and $m$
satisfy the hypotheses \eqref{eq:CobordismConditions}
in Theorem \ref{thm:Cobordism} for the manifold $X_q(n)$.

Because $\La-w\equiv w_2(X_q)$ and $\la_u+w_u\equiv 1\pmod 2$,
the definition \eqref{eq:DefinetwtLa} of
$\tilde\La$ and $\tilde w$
and the equality
$w_2(X_q(n))\equiv w_2(X_q)+\sum_{u=1}^n e_u^*\pmod 2$
imply that $\tilde\La$ and $\tilde w$  satisfy the condition \eqref{eq:CobordismCondition1}
for $X_q(n)$.

The definition \eqref{eq:DefinetwtLa} of $\tilde\La$ also implies that
$\tilde\La^2=\La^2-\sum_{u=1}^n\lambda_u^2$.
Together with \eqref{eq:BlowUpFormulaCondition1}
and the equalities $c(X_q(n))=n+3$ and $\chi_h(X_q(n))=q$
from \eqref{eq:BlownUpExampleChar},
this yields
$$
\tilde\La^2
=
\La^2-\sum_{u=1}^n\la_u^2
>
\delta -c(X_q(n))-4\chi_h(X_q(n)),
$$
so $\tilde\La$ and $\delta$ satisfy \eqref{eq:CobordismCondition2} on $X_q(n)$.

The definition of $\tilde w$ gives $-\tilde w^2=-w^2+\sum_{u=1}^n w_u^2$.
Combining this equality with the assumption \eqref{eq:BlowUpFormulaCondition2}
and the equality $\chi_h(X_q(n))=q$
from \eqref{eq:BlownUpExampleChar},
we obtain
$$
\delta\equiv -\tilde w^2-3\chi_h(X_q(n))\pmod 4,
$$
so $\delta$ and $\tilde w$ satisfy \eqref{eq:CobordismCondition3} on $X_q(n)$.

The condition \eqref{eq:CobordismCondition4} appears directly
as the hypothesis $\delta\ge 2m$ in Lemma \ref{lem:BlowUpCobordism}.
Hence, we can apply Theorem \ref{thm:Cobordism} with
$\tilde\La$, $\tilde w$, $\delta$, and $m$ for the manifold $X_q(n)$.

Because $X_q(n)$ satisfies Witten's Conjecture, we can
apply Lemma \ref{lem:ReduceDFormToB'Sum} to compute the Donaldson invariant
$D^{\tilde w}_{X_q(n)}(h^{\delta-2m}x^m)$.
By \eqref{eq:NoZeroInB}, we have $0\notin B'(X_q(n))$, so $\nu(K_\varphi)=1$
(where $\nu(K)$ is defined in \eqref{eq:DiracSpincFunction}) for all
$\varphi\in (\ZZ/2\ZZ)^n$.
If we abbreviate  the orientation factor
$\eps(\tilde w,K_\varphi)$ in \eqref{eq:DefineOrientationEps} by $\eps(\tilde w,\varphi)$,
use the equality
$\SW_{X_q}'(K)=\SW_{X_q(n)}'(K_\varphi)$
from Theorem \ref{thm:FroyshovSWBlowUp},
and note that by \eqref{eq:BlowUpBasicClasses} the set $B'(X_q(n))$ is enumerated by $(\ZZ/2\ZZ)^n$,
then Lemma \ref{lem:ReduceDFormToB'Sum} implies that
\begin{equation}
\label{eq:CoeffDetLHS1}
\begin{aligned}
{}&
D^{\tilde w}_{X_q(n)}(h^{\delta-2m}x^m)
\\
{}&\quad=
\sum_{\begin{subarray}{l}i+2k\\=\delta-2m\end{subarray}}
\sum_{\varphi\in (\ZZ/2\ZZ)^n}
(-1)^{{\eps(\tilde w,\varphi)}}
\frac{\SW_{X_q}'(K) (\delta-2m)!}{2^{k+n-m} k!i!}
\langle K_\varphi,h\rangle^iQ_{X_q(n)}(h)^k.
\end{aligned}
\end{equation}
In addition, the identity \eqref{eq:CompareCoeff2} in Lemma \ref{lem:ReduceCobordismFormToB'Sum} gives
\begin{equation}
\label{eq:CoeffDetRHS1}
\begin{aligned}
{}&
D^{\tilde w}_{X_q(n)}(h^{\delta-2m}x^m)
\\
{}&\quad
=
\sum_{\begin{subarray}{l}i+j+2k\\=\delta-2m\end{subarray}}
\sum_{\varphi\in (\ZZ/2\ZZ)^n}
(-1)^{{\eps(\tilde w,\varphi)}}
\SW_{X_q}'(K)
\tilde b_{i,j,k}( K_\varphi\cdot\tilde\La)
\langle K_\varphi,h\rangle^i
\langle \tilde\La,h\rangle^j
Q_{X_q(n)}(h)^k.
\end{aligned}
\end{equation}
We will rewrite \eqref{eq:CoeffDetLHS1} and \eqref{eq:CoeffDetRHS1} as sums over terms
of the form
\begin{equation}
\label{eq:BlowUpTerm}
\left\langle K+(-1)^{\pi_1(\varphi)}e_1^*,h\right\rangle^{i_1}
\left(\prod_{u=2}^n\langle e_u^*,h\rangle^{i_u}\right)
\langle \La,h\rangle^j Q_{\widetilde X(n)}(h)^k.
\end{equation}
Using the definition of $K_\varphi$ in \eqref{eq:BasicClassesOnBlowUp},
we expand
$$
\langle K_\varphi,h\rangle^i
=
\left\langle (K+(-1)^{\pi_1(\varphi)}e_1^*)+\sum_{u=2}^n (-1)^{\pi_u(\varphi)}e_u^*,h\right\rangle^i,
$$
and thus
\begin{multline}
\label{eq:ExpandKphi}
\langle K_\varphi,h\rangle^i
\\
=
\sum_{i_1+\cdots+i_n=i}
\binom{i}{i_1,\ \dots , i_n}
(-1)^{\sum_{u=2}^n \pi_u(\varphi)i_u}
\left\langle K+(-1)^{\pi_1(\varphi)}e_1^*,h\right\rangle^{i_1}
\prod_{u=2}^n\langle e_u^*,h\rangle^{i_u}.
\end{multline}
Next, we compute the orientation factors for $\varphi_0:=(0,0,\dots,0)\in (\ZZ/2\ZZ)^n$ to give
\begin{align*}
\eps(\tilde w,\varphi)
{}&\equiv
\frac{1}{2}\left( \tilde w^2 +\tilde w\cdot K_\varphi\right) \pmod 2
\quad\hbox{(by \eqref{eq:DefineOrientationEps})}
\\
{}&\equiv
\frac{1}{2}\left( \tilde w^2 +\tilde w\cdot K_0\right) +\frac{1}{2}(K_\varphi-K_0)\cdot\tilde w \pmod 2
\\
{}&\equiv
\eps(\tilde w,\varphi_0) +\frac{1}{2}(K_\varphi-K_0)\cdot\tilde w \pmod 2
\quad\hbox{(by \eqref{eq:DefineOrientationEps} and \eqref{eq:BasicClassesOnBlowUp})}
\\
{}&\equiv
\eps(\tilde w,\varphi_0)+\sum_{u=1}^n
\frac{1}{2} \left( (-1)^{\pi_u(\varphi)}-1\right)w_u e_u^*\cdot e_u^* \pmod 2
\quad\hbox{(by \eqref{eq:BasicClassesOnBlowUp}),}
\end{align*}
and thus, by \eqref{eq:MultZ2andZ},
\begin{equation}
\label{eq:Orientphi}
\eps(\tilde w,\varphi)
=
\eps(\tilde w,\varphi_0) + w_1\pi_1(\varphi) + \sum_{u=2}^n w_u\pi_u(\varphi) \pmod 2.
\end{equation}
Equations \eqref{eq:ExpandKphi} and \eqref{eq:Orientphi}  imply that
\begin{equation}
\label{eq:BlownUpTerms2}
\begin{aligned}
{}&
(-1)^{\eps(\tilde w,\varphi)}
\langle K_\varphi,h\rangle^i
\\
{}&\quad=
(-1)^{\eps(\tilde w,\varphi_0)+w_1\pi_1(\varphi)}
\sum_{i_1+\cdots+i_n=i}
\binom{i}{i_1,\ \dots , i_n}
(-1)^{\sum_{u=2}^n (i_u+w_u)\pi_u(\varphi)}
\\
{}&\qquad\times
\left\langle K+(-1)^{\pi_1(\varphi)}e_1^*,h\right\rangle^{i_1}
\prod_{u=2}^n\langle e_u^*,h\rangle^{i_u}.
\end{aligned}
\end{equation}
We now split the sum
in the right-hand-side of \eqref{eq:CoeffDetLHS1}
over $(\ZZ/2\ZZ)^n$ into sums over $\pi_1^{-1}(0)$
and $\pi_1^{-1}(1)$:
\begin{equation}
\label{eq:CoeffDetLHS2}
\begin{aligned}
{}&
D^{\tilde w}_{X_q(n)}(h^{\delta-2m}x^m)
\\
{}&\quad
=
\sum_{\begin{subarray}{l}i+2k\\=\delta-2m\end{subarray}}
\frac{\SW_{X_q}'(K)(\delta-2m)!}{2^{k+n-m} k!i!}
Q_{X_q(n)}(h)^k
\\
{}&\qquad\times
\left(
\sum_{\varphi\in \pi_1^{-1}(0) }
(-1)^{{\eps(\tilde w,\varphi)}}
\langle K_\varphi,h\rangle^i
+
\sum_{\varphi\in \pi_1^{-1}(1) }
(-1)^{{\eps(\tilde w,\varphi)}}
\langle K_\varphi,h\rangle^i
\right).
\end{aligned}
\end{equation}
Applying \eqref{eq:BlownUpTerms2} to \eqref{eq:CoeffDetLHS2} yields
\begin{equation}
\label{eq:CoeffDetLHS3}
\begin{aligned}
{}&
D^{\tilde w}_{X_q(n)}(h^{\delta-2m}x^m)
\\
{}&\quad
=
\sum_{\begin{subarray}{c}i_1+\cdots +i_n+2k\\=\delta-2m\end{subarray}}
\frac{\SW_{X_q}'(K) (\delta-2m)!}{2^{k+n-m} k!i_1!\cdots i_n!}
(-1)^{\eps(\tilde w,\varphi_0)}
\left(\prod_{u=2}^n\langle e_u^*,h\rangle^{i_u} \right)
Q_{X_q(n)}(h)^k
\\
{}&\qquad\times
\left(
\sum_{\varphi\in \pi_1^{-1}(0) }
(-1)^{\sum_{u=2}^n \pi_u(\varphi)(w_u+i_u)}
\langle K+ e_1^*,h\rangle^{i_1}
\right.
\\
{}&\qquad\quad
\left.
+(-1)^{w_1}
\sum_{\varphi\in \pi_1^{-1}(1) }
(-1)^{\sum_{u=2}^n \pi_u(\varphi)(w_u+i_u)}
\langle K- e_1^*,h\rangle^{i_1}
\right).
\end{aligned}
\end{equation}
Identifying $\pi_1^{-1}(0)$ and $\pi_1^{-1}(1)$
with $(\ZZ/2\ZZ)^{n-1}$ and applying Lemma \ref{lem:PermutationSumAsDifferenceOperator}
and the definition \eqref{eq:DefineSumFactor1} of $p^{\tilde w}(i_2,\dots,i_n)$
yields, for $a=0,1$,
$$
\sum_{\pi_1^{-1}(a)}
(-1)^{\sum_{u=2}^n \pi_u(\varphi)(w_u+i_u)}
=
p^{\tilde w}(i_2,\dots,i_n).
$$
Thus, we may rewrite \eqref{eq:CoeffDetLHS3} as
\begin{equation}
\label{eq:CoeffDetLHS4}
\begin{aligned}
{}&
\frac{(-1)^{\eps(\tilde w,\varphi_0)}}{\SW'_{X_q}(K)}
D^{\tilde w}_{X_q(n)}(h^{\delta-2m}x^m)
\\
{}&\quad
=
\sum_{\begin{subarray}{c}i_1+\cdots i_n+2k\\=\delta-2m\end{subarray}}
\frac{ (\delta-2m)!}{2^{k+n-m} k!i_1!\cdots i_n!}
p^{\tilde w}(i_2,\dots,i_n)
\left(\prod_{u=2}^n\langle e_u^*,h\rangle^{i_u} \right)
Q_{X_q(n)}(h)^k
\\
{}&\qquad\times
\left(
\langle K+ e_1^*,h\rangle^{i_1}
+
(-1)^{w_1}
\langle K- e_1^*,h\rangle^{i_1}
\right).
\end{aligned}
\end{equation}
We next split the sum over $(\ZZ/2\ZZ)^n$ on the right-hand-side of \eqref{eq:CoeffDetRHS1}
into sums over $\pi_1^{-1}(0)$ and $\pi_1^{-1}(1)$:
\begin{equation}
\label{eq:CoeffDetRHS2}
\begin{aligned}
D^{\tilde w}_{X_q(n)}(h^{\delta-2m}x^m)
&=
\sum_{\begin{subarray}{l}i+j+2k\\=\delta-2m\end{subarray}}
SW'_{X_q}(K)
\langle \tilde\La,h\rangle^j
Q_{X_q(n)}(h)^k
\\
{}&\qquad
\times\left(
\sum_{\varphi\in \pi_1^{-1}(0) }
(-1)^{{\eps(\tilde w,\varphi)}}
\tilde b_{i,j,k}( K_\varphi\cdot\tilde\La)
\langle K_\varphi,h\rangle^i
\right.
\\
{}&\qquad\quad
\left.
+
\sum_{\varphi\in \pi_1^{-1}(1) }
(-1)^{{\eps(\tilde w,\varphi)}}
\tilde b_{i,j,k}( K_\varphi\cdot\tilde\La)
\langle K_\varphi,h\rangle^i
\right).
\end{aligned}
\end{equation}
We rewrite the argument $K_\varphi\cdot\La$ in the coefficient $\tilde b_{i,j,k}$,
\begin{align*}
K_\varphi\cdot\tilde\La
{}&=
K_0\cdot\tilde\La+(K_\varphi-K_0)\cdot\tilde\La
\\
{}&=
K_0\cdot\tilde\La +\sum_{u=1}^n ((-1)^{\pi_u(\varphi)}-1)\la_u (e_u^*\cdot e_u^*),
\end{align*}
and thus by \eqref{eq:MultZ2andZ},
\begin{equation}
\label{eq:VariationOfKLa}
K_\varphi\cdot\tilde\La
=
K_0\cdot\tilde\La+2\pi_1(\varphi)\la_1 + 2\sum_{u=2}^n\pi_u(\varphi)\la_u.
\end{equation}
Substituting \eqref{eq:BlownUpTerms2} and \eqref{eq:VariationOfKLa} into \eqref{eq:CoeffDetRHS2},
together with the definitions \eqref{eq:DefineOrientationEps} of $\eps(\tilde w,K_\varphi) \equiv \eps(\tilde w,\varphi)$ and \eqref{eq:BasicClassesOnBlowUp} of $K_0$,
yields
\begin{equation}
\label{eq:CoeffDetRHS3}
\begin{aligned}  
{}&\frac{(-1)^{\eps(\tilde w,\varphi_0)}}{\SW'_{X_q}(K)}
D^{\tilde w}_{X_q(n)}(h^{\delta-2m}x^m)
\\
&\quad =
\sum_{\begin{subarray}{c}i_1+\cdots +i_n+j+2k\\=\delta-2m\end{subarray}}
\binom{i_1+\cdots+i_n}{i_1,\ \dots , i_n}
\left(\prod_{u=2}^n \langle e_u^*,h\rangle^{i_u}\right)
\langle \tilde\La,h\rangle^j
Q_{X_q(n)}(h)^k
\\
&\qquad \times
\left(
\sum_{\varphi\in \pi_1^{-1}(0) }
(-1)^{\theta_\varphi}
\tilde b_{i,j,k}\left( K_0\cdot\tilde\La  + 2\sum_{u=1}^n \pi_u(\varphi)\la_u \right)
\langle K+e_1^*,h\rangle^{i_1}
\right.
\\
&\qquad 
\left.
+ (-1)^{w_1}\sum_{\varphi\in \pi_1^{-1}(1) }
(-1)^{\theta_\varphi}
\tilde b_{i,j,k}\left( K_0\cdot\tilde\La +2\la_1 + 2\sum_{u=1}^n \pi_u(\varphi)\la_u \right)
\langle K-e_1^*,h\rangle^{i_1}
\right),
\end{aligned}
\end{equation}
where we write $\theta_\varphi$ above for
$$
\theta_\varphi := \sum_{u=2}^n \pi_u(\varphi)(w_u+i_u).
$$
By Lemma \ref{lem:PermutationSumAsDifferenceOperator},
\begin{align*}
{}&
\sum_{\varphi\in (\ZZ/2\ZZ)^{n-1}}
(-1)^{\sum_{u=2}^n \pi_u(\varphi)(w_u+i_u)}
\tilde b_{i,j,k}\left(K_0\cdot\tilde\La +2\pi_1(\phi)\la_1 + 2\sum_{u=1}^n \pi_u(\varphi)\la_u\right)
\\
{}&\quad=
\nabla^{i_2+w_2}_{2\la_2}\cdots \nabla^{i_n+w_n}_{2\la_n} \tilde b_{i,j,k}\left(K_0\cdot\tilde\La+2\pi_1(\phi)\la_1\right).
\end{align*}
Substituting the preceding equality into \eqref{eq:CoeffDetRHS3} yields
\begin{equation}
\label{eq:CoeffDetRHS4}
\begin{aligned}
{}&
\frac{(-1)^{\eps(\tilde w,\varphi_0)}}{\SW'_{X_q}(K)}
D^{\tilde w}_{X_q(n)}(h^{\delta-2m}x^m)
\\
{}&\quad
=
\sum_{\begin{subarray}{c}i_1+\cdots i_n+j+2k\\=\delta-2m\end{subarray}}
\binom{i_1+\cdots+i_n}{i_1,\ \dots , i_n}
\left(\prod_{u=2}^n \langle e_u^*,h\rangle^{i_u}\right)
\langle \tilde\La,h\rangle^j
Q_{X_q(n)}(h)^k
\\
{}&\qquad\times
\bigg(
\nabla^{i_2+w_2}_{2\la_2}\cdots \nabla^{i_n+w_n}_{2\la_n} \tilde b_{i,j,k}(K_0\cdot\tilde\La)
\langle K+e_1^*,h\rangle^{i_1}
\\
{}&\qquad\quad+
(-1)^{w_1}
\nabla^{i_2+w_2}_{2\la_2}\cdots \nabla^{i_n+w_n}_{2\la_n} \tilde b_{i,j,k}(K_0\cdot\tilde\La+2\la_1)
\langle K-e_1^*,h\rangle^{i_1}
\bigg).
\end{aligned}
\end{equation}
Comparing
equations \eqref{eq:CoeffDetLHS4} and \eqref{eq:CoeffDetRHS4}
gives the desired equality \eqref{eq:BlownUpUsefulCobordismFormula}.
\end{proof}

We now review a result giving the coefficients $\tilde b_{i,j,k}$ for $i\ge c(X)-3$.

\begin{prop}
\label{prop:HighDegreeCoefficients}
\cite[Proposition 4.8]{FL6}
Let $n>0$ and $q\ge 2$ be integers.
If $x,y$ are integers and $i,j,k,m$ are non-negative integers
satisfying, for $A:=i+j+2k+2m$,
\begin{subequations}
\begin{align}
\label{eq:HighDegreeCondition1}
i& \ge n,
\\
\label{eq:HighDegreeCondition2}
y& > A-4q-3-n,
\\
\label{eq:HighDegreeCondition3}
A & \ge 2m,
\\
\label{eq:HighDegreeCondition4}
x& \equiv y\equiv 0\pmod 2,
\end{align}
\end{subequations}
then the coefficients
$\tilde b_{i,j,k}(\chi_h,c_1^2,\La\cdot K,\La^2,m)$
defined in \eqref{eq:OrientedCoeff} are given by
$$
\tilde b_{i,j,k}(q,q -3-n,x,y,m)
=
\begin{cases}
\displaystyle\frac{(A-2m)!}{k!i!}2^{m-k-n}  & \text{if $j=0$,}
\\
0 & \text{if $j>0$.}
\end{cases}
$$
\end{prop}

\begin{rmk}
The expression for the coefficients $\tilde b_{i,j,k}$ given in Proposition \ref{prop:HighDegreeCoefficients} differs
from that given for the coefficients $b_{i,j,k}$ in \cite[Proposition 4.8]{FL6}
exactly by the factor of $(-1)$ appearing in the definition  \eqref{eq:OrientedCoeff}.
\end{rmk}

Because of
the condition \eqref{eq:HighDegreeCondition1}, Proposition
\ref{prop:HighDegreeCoefficients}  only determines the coefficients
$\tilde b_{i,j,k}$ with $i\ge c(X)-3$.
We next derive a difference equation satisfied by
the coefficients $\tilde b_{i,j,k}$ with
$1\le i<c(X)-3$.

\begin{prop}
\label{prop:LowICoefficientDifferenceRelation}
Let $n>1$ and $q\ge 2$ be integers.
If $x,y$ are integers and $p,j,k,m$ are non-negative integers
satisfying, for $A:=p+j+2k+2m$,
\begin{subequations}
\begin{align}
\label{eq:DiffRelatCond0}
1&\le p\le n-1,
\\
\label{eq:DiffRelatCond1}
y&>A-4q-n-3,
\\
\label{eq:DiffRelatCond2}
y&\equiv A - (n+3)\pmod 4,
\\
\label{eq:DiffRelatCond3}
x-y&\equiv 0\pmod 2,
\end{align}
\end{subequations}
and we abbreviate
$$
\tilde b_{p,j,k}(x) = \tilde b_{p,j,k}(q,q-n-3,x,y,m),
$$
then
\begin{equation}
\label{eq:LowICoefficientDifferenceRelation}
\left(
\nabla^1_4
\right)^{n-p}\tilde  b_{p,j,k}(x)
= 0.
\end{equation}
\end{prop}

\begin{proof}
Let $X_q(n)$ be the
manifold defined in \eqref{eq:BlownUpExampleNotation}.
By \eqref{eq:BlownUpExampleChar}, we have $\chi_h(X_q(n))=q$ and $c_1^2(X_q(n))=q-n-3$.
We will  apply Lemma \ref{lem:AlgCoeff} to equation
\eqref{eq:BlownUpUsefulCobordismFormula} for the manifold $X_q(n)$.
Let $f_1,f_2\in H^2(X_q;\ZZ)\subset H^2(X_q(n);\ZZ)$
and $K\in B(X_q)$ be the cohomology classes
appearing in the properties of $X_q$ listed
at the beginning of Section \ref{subsec:ExManifoldsBlowUp},
satisfying $f_i\cdot K=0$ and $f_i^2=0$ for $i=1,2$ and $f_1\cdot f_2=1$.
For $y_0:=\thalf\left( y+(x+2(n-p))^2+4(n-p)\right)$, define
\begin{equation}
\label{eq:Expression_tilde_Lambda_as_Lambda_plus_sum_lambdau_eu*}
\tilde\La=\La+\sum_{u=1}^n \la_ue_u^*,
\end{equation}
where we define $\La:=y_0f_1+f_2\in H^2(X_q;\ZZ)$ and the non-negative integers $\lambda_u$ are given by
\begin{equation}
\label{eq:lauInTildeLa}
\la_u
=
\begin{cases}
-(x+2(n-p)) & \text{if $u=1$},
\\
0 & \text{if $1<u\le p$},
\\
2 & \text{if $p+1\le u\le n$}.
\end{cases}
\end{equation}
The assumption \eqref{eq:DiffRelatCond3} that $y\equiv x\pmod 2$
implies that
$y_0$
is an integer.
Thus, for $K_0$ as in
\eqref{eq:BasicClassesOnBlowUp}, we see that
\begin{equation}
\label{eq:La2InDiff}
\tilde\La^2=y
\quad\hbox{and}\quad
\tilde\La\cdot K_0=x.
\end{equation}
Define $\tilde w:=\tilde\La-K_0$,
where $K_0$ is as in \eqref{eq:BasicClassesOnBlowUp}.
We claim that $\tilde w,\tilde\La$
and $\delta:=A$ satisfy the hypotheses
of Lemma \ref{lem:BlowUpCobordism}.
We have
\begin{equation}
\label{eq:tilde_Lambda_minus_tilde_w_is_w2(Xq(n))_mod_2}
\tilde\La-\tilde w=K_0\equiv w_2(X_q(n))\pmod 2
\end{equation}
by construction
and $\delta\ge 2m$ by definition.
%
%
By \eqref{eq:Expression_tilde_Lambda_as_Lambda_plus_sum_lambdau_eu*}
\begin{align*}
\La^2 {}&= \tilde\La^2+\sum_{u=1}^n\la_u^2
\\
{}&=y +\sum_{u=1}^n\la_u^2
\quad\text{(by  \eqref{eq:La2InDiff})}
\\
{}&> \delta-(n+3)-4q  +\sum_{u=1}^n\la_u^2
\quad\text{(by \eqref{eq:DiffRelatCond1} and $\delta=A$)},
\end{align*}
so the condition \eqref{eq:BlowUpFormulaCondition1} holds.

For $\tilde w=\tilde\La-K_0$ as above, we can write
\begin{equation}
\label{eq:Definition_tilde_w_as_w_plus_sum_wu_eu*}
\tilde w=w+\sum_{u=1}^n w_u e_u^*,
\end{equation}
where $w\in H^2(X_q;\ZZ)$.
To verify that the hypothesis
\eqref{eq:BlowUpFormulaCondition2}
in Lemma \ref{lem:BlowUpCobordism} holds, we compute
\begin{align*}
\tilde w^2
{}&= \tilde\La^2  -2K_0\tilde \La +K_0^2
\\
{}&\equiv K_0^2 -\tilde\La^2\pmod 4
\quad\text{(as $-2K_0\tilde\La\equiv -2\tilde\La^2\pmod 4$, since $K_0$ characteristic)}
\\
{}&\equiv (q-n-3) -\tilde\La^2\pmod 4
\quad\text{(by \eqref{eq:BlownUpExampleChar} and \eqref{eq:SWSimpleTypeOnBlowUp})}
\\
{}&\equiv  (q-n-3)-\delta+(n+3)\pmod 4
\quad\text{(by \eqref{eq:DiffRelatCond2} and $\delta=A$)}
\\
{}&\equiv -\delta-3q\pmod 4.
\end{align*}
Combining the preceding equality with $w^2=\tilde w^2+\sum_{u=1}^n w_u^2$
yields $w^2\equiv -\delta-3q+\sum_{u=1}^n w_u^2\pmod 4$ and so
condition \eqref{eq:BlowUpFormulaCondition2} holds. Hence, we can
apply Lemma \ref{lem:BlowUpCobordism} with the given values for
$\tilde\La$, $\tilde w$, $\delta$, and $m$ to the coefficients $\tilde b_{p,j,k}(x)$.

Next, we claim that the set $\{K+e_1^*,K-e_1^*,e_2^*,\dots,e_n^*,\tilde \La,Q_{X_q(n)}\}$
is algebraically independent in the sense of Lemma \ref{lem:AlgCoeff}.
To see that $K+e_1^*,K-e_1^*,e_2^*,\dots,e_n^*,\tilde \La$ are linearly independent,
assume there is a linear combination with $a,b,c,d_2,\dots,d_n\in\RR$,
\begin{equation}
\label{eq:LinComb}
a(K+e_1^*)+b(K-e_1^*)+c\tilde\La +\sum_{u=2}^n d_u e_u^*=0\in H^2(X_q(n);\RR).
\end{equation}
Because there is a  direct sum decomposition,
$$
H^2(X_q(n);\RR)\cong H^2(X_q;\RR)\oplus \bigoplus_{u=1}^n\RR e_u^*,
$$
the equality \eqref{eq:LinComb} gives
\begin{subequations}
\begin{align}
\label{eq:LinIndepDirectSum1}
(a+b)K+c\La&=0\in H^2(X_q;\RR),
\\
\label{eq:LinIndepDirectSum2}
(a-b)e_1^*+c(\tilde\La-\La)+\sum_{u=2}^n d_u e_u^*&=0\in \bigoplus_{u=1}^n\RR e_u^*.
\end{align}
\end{subequations}
By \eqref{eq:HyperbolicSummand2}, the classes
$K$ and $\La=y_0f_1+f_2$ in $H^2(X_q;\RR)$
are linearly
independent because $K,f_1,f_2$ are linearly
independent in $H^2(X_q;\RR)$.  Thus,
\eqref{eq:LinIndepDirectSum1} implies that
$a+b=0$ and $c=0$.  Equation \eqref{eq:LinComb} then reduces to
$$
a(K+e_1^*)-a(K-e_1^*)+\sum_{u=2}^n d_u e_u^*
=
2a e_1^*+\sum_{u=2}^n d_u e_u^*=0.
$$
By the linear independence of $e_1^*,\dots,e_n^*$, we have $a=-b=0$ and $d_1=\cdots=d_n=0$, proving
the linear independence of $K+e_1^*,K-e_1^*,e_2^*,\dots,e_n^*,\tilde \La$.
Next, we observe that  the intersection of kernels,
$$
\bK_1:=\Ker (K+e_1^*)\cap \Ker (K-e_1^*)
\cap \Ker \tilde\La
\cap \bigcap_{u=2}^n \Ker e_u^*
\subset H_2(X_q(n);\RR),
$$
contains the intersection of kernels
$$
\bK_2:=\Ker K\cap \Ker f_1\cap \Ker f_2\subset H_2(X_q;\RR).
$$
Because the restriction of $Q_{X_q(n)}$ to $\bK_2$ equals the restriction
of $Q_{X_q}$ to $\bK_2$ and the restriction of $Q_{X_q}$ to $\bK_2$
is non-zero by \eqref{eq:HyperbolicSummand3}, the restriction of $Q_{X_q}$ to $\bK_1$
is also non-zero.  Thus, Lemma \ref{lem:AlgCoeff} implies that
the set $\{K+e_1^*,K-e_1^*,e_2^*,\dots,e_n^*,\tilde \La,Q_{X_q(n)}\}$
is algebraically independent.

This algebraic independence  and
Lemma \ref{lem:AlgCoeff} imply that
the coefficients of the term
\begin{equation}
\label{eq:ComparisonTerm}
\langle K+e_1^*,h\rangle^{i_1} \prod_{u=2}^n\langle e_u^*,h\rangle^{i_u}
\langle \La,h\rangle^j Q_{\widetilde X(n)}(h)^k
\end{equation}
on the left and right-hand sides of the identity \eqref{eq:BlownUpUsefulCobordismFormula} in Lemma \ref{lem:BlowUpCobordism} will be equal.  In particular, we consider the term \eqref{eq:ComparisonTerm} where
\begin{equation}
\label{eq:Values_iu_for_2_leq_n}
i_1=\cdots=i_p=1,\quad i_{p+1}=\cdots=i_{n}=0.
\end{equation}
The coefficient of this term on the left-hand-side of \eqref{eq:BlownUpUsefulCobordismFormula}
is given by a multiple of the
expression $p^{\tilde w}(i_2,\dots,i_n)$ defined in \eqref{eq:DefineSumFactor1}.
By the definition $\tilde w=\tilde\La-K_0$, the definition \eqref{eq:lauInTildeLa} that $\la_n=2$,
and the assumption \eqref{eq:Values_iu_for_2_leq_n} that $i_n=0$, 
we see that $w_n+i_n\equiv 1\pmod 2$, so $p^{\tilde w}(i_2,\dots,i_n)=0$ and the
coefficient of this term on the left-hand-side of \eqref{eq:BlownUpUsefulCobordismFormula} vanishes.
By the definition $\tilde w=\tilde\La-K_0$, the definition \eqref{eq:lauInTildeLa} of $\la_u$,
and the assumption \eqref{eq:Values_iu_for_2_leq_n} on the values of $i_u$, we see that
$w_u+i_u\equiv 0\pmod 2$ for $u=2,\dots,p$ and $w_u+i_u\equiv 1\pmod 2$ for $u=p+1,\dots,n$.
Hence, the coefficient of the term \eqref{eq:ComparisonTerm} satisfying
\eqref{eq:Values_iu_for_2_leq_n} on the right-hand-side
of \eqref{eq:BlownUpUsefulCobordismFormula} is
\begin{equation}
  \label{eq:LowOrderRHSTermCoeff}
p!\left(\nabla^0_0\right)^{p-1}
\left(\nabla^1_4\right)^{n-p}
 \tilde b_{p,j,k}(x)
 =
p! 2^{p-1}\left(\nabla^1_4\right)^{n-p}
 \tilde b_{p,j,k}(x).
\end{equation}
Because
Lemma \ref{lem:AlgCoeff}
implies that the coefficients of the term \eqref{eq:ComparisonTerm} on
the left and right-hand sides of \eqref{eq:BlownUpUsefulCobordismFormula}
are equal, the expression given by the right-hand-side of \eqref{eq:LowOrderRHSTermCoeff} must also vanish,
giving the desired result.
\end{proof}

\begin{rmk}
We required $p\ge 1$ in Proposition \ref{prop:LowICoefficientDifferenceRelation} because, in order to get information
about the coefficients $\tilde b_{0,j,k}$, we would have to consider the term
$$
\langle \La,h\rangle^j Q_{X_q(n)}^k
$$
in the equality \eqref{eq:BlownUpUsefulCobordismFormula}.
The coefficient of this term on the right-hand-side of \eqref{eq:BlownUpUsefulCobordismFormula}
is a multiple of
$$
\nabla^{w_1}_{2\la_1}\dots\nabla^{w_2}_{2\la_n}\tilde b_{0,j,k}(x_0),
$$
and so the argument of Proposition \ref{prop:LowICoefficientDifferenceRelation} would
show that $\tilde b_{0,j,k}$ also satisfies a difference equation of degree $n$.
However, the choice of $\la_1$ in \eqref{eq:lauInTildeLa} interacted with the possible
values of $x_0$, complicating the use of this result.
By Lemma \ref{lem:BlownUpCobordism}, we can avoid the need to pursue this argument.
\end{rmk}

Proposition \ref{prop:LowICoefficientDifferenceRelation} and the result for difference
equations given by Corollary \ref{cor:DegreeNPolyn} allow us to write the coefficients
$\tilde b_{i,j,k}$ as polynomials on $H_2(X;\RR)$.  We will combine this
fact with Lemma \ref{lem:SCSTVanishingSum} to show that, for manifolds of superconformal
simple type, the coefficients $\tilde b_{i,j,k}$ with $i\le c(X)-4$ do not contribute
to the expression for the Donaldson invariant in \eqref{eq:CompareCoeff2}.

\begin{cor}
\label{cor:PolynomialDependence}
Continue the assumptions of Proposition \ref{prop:LowICoefficientDifferenceRelation}.
In addition assume
\begin{enumerate}
\item
There is a class $K_1\in B(X)$ such that $\La\cdot K_1=0$;
\item
For all $K\in B(X)$, we have $\La\cdot K\equiv 0\pmod 4$.
\end{enumerate}
Then
for $1\le i\le n-1$, the function
$\tilde b_{i,j,k}$ is a  polynomial of degree $n-1-i$  in $\La\cdot K$ and thus
\begin{equation}
\label{eq:PolynDependenceCoeff}
\tilde b_{i,j,k}(q,q-n-3,K\cdot \La,\La^2,m)
=
\sum_{u=0}^{n-1-i}
\tilde b_{u,i,j,k}(q,q-n-3,\La^2,m)\langle K,h_\La\rangle^u,
\end{equation}
where $h_\La=\PD[\La]$ is the Poincar\'e dual of $\La$  and
if $u\equiv n+i\pmod 2$, then
\begin{equation}
\label{eq:ParityOfVanishingbuCoeff}
\tilde b_{u,i,j,k}(q,q-n-3,\La^2,m)=0.
\end{equation}
\end{cor}

\begin{proof}
The Poincar\'e dual $h_\La$ has the property that $\langle K,h_\La\rangle=K\cdot\La$
for any $K\in H^2(X;\ZZ)$.
The assumption $\La\cdot K\equiv 0\pmod 4$ implies that
it is enough to compute $\tilde b_{i,j,k}(q,q-n-3,4x,\La^2,m)$ for $x\in\ZZ$.
Equation \eqref{eq:PolynDependenceCoeff}  then
follows from equation \eqref{eq:LowICoefficientDifferenceRelation} in
Proposition \ref{prop:LowICoefficientDifferenceRelation}
and Corollary \ref{cor:DegreeNPolyn}.
Because $\La\cdot K\equiv 0\pmod 4$,
equation \eqref{eq:SignChangeOfTildeBCoeff}
implies that
$$
\tilde b_{i,j,k}(q,q-n-3,-K\cdot \La,\La^2,m)
=
(-1)^{n+3+i}\tilde b_{i,j,k}(q,q-n-3,K\cdot \La,\La^2,m).
$$
Therefore, the coefficients $\tilde b_{u,i,j,k}$ in \eqref{eq:PolynDependenceCoeff}
with $u\not\equiv n+3+i\pmod 2$, or equivalently
$u\equiv n+i\pmod 2$ vanish as asserted in \eqref{eq:ParityOfVanishingbuCoeff}.
\end{proof}

\begin{rmk}
We can remove the assumption in Corollary \ref{cor:PolynomialDependence}
that there is a class $K_1\in B(X)$ with $K_1\cdot\La=0$ but then the coefficient
will be given as a polynomial in the variable $\langle K-K_1,h_\La\rangle$
which is less convenient for the computations in the proof of
Theorem \ref{thm:SCSTImpliesWC}.
\end{rmk}

\section{Proofs of main results}
\label{sec:MainProof}
We begin by establishing
the following  algebraic consequence of superconformal simple type;
this will allow us to  show that Witten's Conjecture \ref{conj:WC} holds even
without determining the coefficients
$\tilde b_{i,j,k}$ with $i<c(X)-3$.

\begin{lem}
\label{lem:SCSTVanishingSum}
Let $X$ be a standard four-manifold of superconformal simple type.
Assume that $0\notin B(X)$.
If $w\in H^2(X,\ZZ)$ is characteristic
and
$j,u\in\NN$ satisfy
$j+u< c(X)-3$ and $j+u\equiv c(X)\pmod 2$, then
\begin{equation}
\label{eq:SWPolyVanishingDeriv}
\sum_{K\in B'(X)} (-1)^{\eps(w,K)}\SW_X'(K)\langle K,h_1\rangle^j\langle K,h_2\rangle^u
=
0,
\end{equation}
for any $h_1,h_2\in H_2(X;\RR)$.
\end{lem}

\begin{proof}
Let $i=j+u$.  Because $i\le c(X)-4$ by hypothesis,
the function $\SW_X^{w,i}:H_2(X;\RR)\to \RR$ vanishes identically
by the defining property \eqref{eq:SCST} of superconformal simple type and thus
\begin{equation}
\label{eq:SWPolyDerivVanishing}
\left.\frac{\rd^i}{\rd s^j\rd t^u}
\SW_X^{w,i}(s h_1+th_2)\right|_{s=t=0}=0.
\end{equation}
Substituting the equality
\begin{align*}
\left.
\frac{\rd^i}{\rd s^j\rd t^u}
\langle K, s h_1 +t h_2\rangle^i\right|_{s=t=0}
{}&=
\left.
\frac{\rd^i}{\rd s^j\rd t^u}
\sum_{a+b=i} \binom{i}{a} s^at^b \langle K,h_1\rangle^a \langle K,h_2\rangle^b
\right|_{s=t=0}
\\
{}&=
i!\langle K,h_1\rangle^j \langle K,h_2\rangle^u
\end{align*}
into the equality \eqref{eq:SWPolyDerivVanishing} and using the expression
in \eqref{eq:SCST} for $\SW^{w,i}_X$ yields
\begin{align*}
0{}&=
\left.\frac{\rd^i}{\rd s^j\rd t^u}
\SW_X^{w,i}(s h_1+th_2)\right|_{s=t=0}
\\
{}&=
\sum_{K\in B(X)} (-1)^{\eps(w,K)}\SW_X'(K)
\left.
\frac{\rd^i}{\rd s^j\rd t^u}
\langle K, s h_1 +t h_2\rangle^i\right|_{s=t=0}
\\
{}&=
i!\sum_{K\in B(X)} (-1)^{\eps(w,K)}\SW_X'(K)
\langle K,h_1\rangle^j \langle K,h_2\rangle^u.
\end{align*}
This proves that
\begin{equation}
\label{eq:SWPolyVanishingDerivBSum}
0
=
\sum_{K\in B(X)} (-1)^{\eps(w,K)}\SW_X'(K)
\langle K,h_1\rangle^j \langle K,h_2\rangle^u.
\end{equation}
Because $\SW_X'(K)=(-1)^{\chi_h(X)}\SW_X'(-K)$ by \cite[Corollary 6.8.4]{MorganSWNotes},
the terms in \eqref{eq:SWPolyVanishingDerivBSum}
corresponding to $K$ and $-K$, namely
$$
(-1)^{\frac{1}{2} (w^2+w\cdot K)}\SW_X'(K)\langle K,h_1\rangle^j \langle K,h_2\rangle^u
$$
and
$$
(-1)^{\frac{1}{2} (w^2-w\cdot K)}\SW_X'(-K)\langle -K,h_1\rangle^j \langle -K,h_2\rangle^u
$$
differ by the sign
$$
(-1)^{\chi_h(X) +w\cdot K+j+u}.
$$
Because $w$ is characteristic
and because $X$ has Seiberg--Witten simple type, we have
$w\cdot K\equiv K^2\equiv c_1^2(X)\pmod 2$.  Hence,
\begin{align*}
\chi_h(X)+w\cdot K + j+u
&\equiv
\chi_h(X)+c_1^2(X)+j+u
  \\
  &\equiv
c(X)+j+u\pmod 2.
\end{align*}
Hence,
the assumptions that $j+u\equiv c(X)\pmod 2$ and $0\notin B(X)$
imply that the terms  in \eqref{eq:SWPolyVanishingDerivBSum}
corresponding to $K$ and $-K$ are  equal.
Because $0\notin B(X)$, $K\neq -K$ for all $K\in B(X)$ and so
by combining these terms, we can rewrite \eqref{eq:SWPolyVanishingDerivBSum} as
\begin{equation}
\label{eq:SWPolyVanishingDerivB'Sum}
0
=
2\sum_{K\in B'(X)} (-1)^{\eps(w,K)}\SW_X'(K)
\langle K,h_1\rangle^j \langle K,h_2\rangle^u,
\end{equation}
which yields the desired result.
\end{proof}

The following lemma allows us to apply Corollary \ref{cor:PolynomialDependence}.

\begin{lem}
\label{lem:PositiveOnComplement}
Let $X$ be a standard four-manifold with odd intersection form.  Then for any $K\in B(X)$, there
is a class $\La\in H^2(X;\ZZ)$ with $\La^2>0$ and $\La\cdot K=0$.
\end{lem}

\begin{proof}
Because $Q_X$ is odd and $b^+(X)\ge 3$, by \cite[Theorem 1.2.21]{GompfStipsicz} we can write
$$
H^2(X;\ZZ)
\cong\left( \oplus_{i=1}^m \ZZ e_i\right) \oplus \left(\oplus_{j=1}^n \ZZ e_j\right),
$$
where $m\ge 3$, and
$Q_X$ is diagonal with respect to the basis $\{e_1,\dots,e_m,f_1,\dots,f_n\}$, where
$e_i^2=1$, and $f_j^2=-1$.  Because $K\in B(X)$ is characteristic, we can
write
$$
K=\sum_{i=1}^m a_i e_i + \sum_{j=1}^n b_jf_j,
$$
where $a_i\equiv 1\pmod 2$.  Define $\La:=a_2e_1-a_1e_2$.  Then $\La\cdot K=0$ and
$\La^2=a_1^2+a_2^2>0$ as required.
\end{proof}

Corollary \ref{cor:PolynomialDependence} and Lemma \ref{lem:SCSTVanishingSum}
provide the basis of the proof of our main result:

\begin{proof}[Proof of Theorem \ref{thm:SCSTImpliesWC}]
By Theorem \ref{thm:WCBlowDownInvariance}, we may blow up
$X$ without loss of generality.
According to Lemma \ref{lem:SCSTBlowUp}, the superconformal simple type condition is preserved
under blow-up. If $\widetilde X$ is the blow-up of $X$, then
the characterization of $B(\widetilde X)$ in \eqref{eq:SWBasicsOfBlowUp} implies that
$0\notin B(\widetilde X)$.  Thus, by replacing $X$ with its blow-up if necessary,
we may assume without loss of generality that $c_1^2(X)\neq 0$, $Q_X$ is odd,
$c(X)\ge 5$, $0\notin B(X)$ and $\nu(K)=1$, where $\nu(K)$ is defined in \eqref{eq:DiracSpincFunction}
for each $K\in B(X)$.

By Proposition \ref{prop:IndepOfWCFromw}, it suffices
to prove that  equation \eqref{eq:DInvarForWCB'Sum}
in Lemma \ref{lem:ReduceDFormToB'Sum} holds when $w\in H^2(X;\ZZ)$ is
characteristic.
Because $w$ is characteristic,
\begin{align*}
w^2 {}&\equiv \si(X) \pmod 8 \quad\text{(by \cite[Lemma 1.2.20]{GompfStipsicz})}
\\
{}&= c_1^2(X)-8\chi_h(X)\quad\text{(by \eqref{eq:CharNumbers})}
\\
{}&\equiv c_1^2(X)\pmod 8.
\end{align*}
Thus, $D^w_X(h^{\delta-2m}x^m)=0$ unless
$$
  \delta \equiv -w^2-3\chi_h(X)\equiv \chi_h(X)-c_1^2(X)-4\chi_h(X)\equiv c(X)\pmod 4
$$
and we need only compute the Donaldson invariant
$D^w_X(h^{\delta-2m}x^m)$, where
\begin{equation}
\label{eq:deltaAndmAssumption}
\delta\ge 2m\quad \text{and}\quad \delta \equiv -w^2-3\chi_h(X)\equiv c(X)\pmod 4.
\end{equation}
To apply Lemma \ref{lem:BlownUpCobordism} to compute $D^w_X(h^{\delta-2m}x^m)$,
we abbreviate
\begin{equation}
\label{eq:CoefficientAbbreviation}
\tilde b_{i,j,k}(\La\cdot K)
=
\tilde b_{i,j,k}(\chi_h(X),c_1^2(X)-1,\La\cdot K,\La^2,m),
\end{equation}
and verify that we can find $\La\in H^2(X;\ZZ)$ satisfying the conditions of
Theorem \ref{thm:Cobordism} and hence those of Lemma \ref{lem:BlownUpCobordism}
as well as of Corollary \ref{cor:PolynomialDependence}.

By Lemma \ref{lem:PositiveOnComplement} and our observation that by replacing $X$
with its blow-up if necessary we can assume that
$Q_X$ is odd and there are classes
$K_0\in B(X)$ and $\La_0\in H^2(X;\ZZ)$ with
$\La_0^2>0$ and $\La_0\cdot K_0=0$.
Because any $K\in B(X)$ can be written as $K=K_0+2 L_K$ for $L_K\in H^2(X;\ZZ)$,
if $\La=2b\La_0$ where $b\in\NN$, then
\begin{equation}
\label{eq:CorPolDepAssumptions}
K_0\cdot\La=0\quad\text{and}\quad
K\cdot\La\equiv 0\pmod 4\ \text{for all $K\in B(X)$},
\end{equation}
so $\La$ satisfies two of the assumptions of Corollary \ref{cor:PolynomialDependence}.

If $w\in H^2(X;\ZZ)$ is characteristic and $\La=2b\La_0$, where
$b\in\NN$ and $\La_0^2>0$, then $\La-w\equiv w_2(X)\pmod 2$ and so
condition \eqref{eq:CobordismCondition1} holds.
Given $\delta$, by choosing $b$  sufficiently large,
we can ensure
\begin{equation}
\label{eq:La2Assumption}
\La^2+c(X)+4\chi_h(X)>\delta,
\end{equation}
so condition \eqref{eq:CobordismCondition2} holds.  Conditions
\eqref{eq:CobordismCondition3} and \eqref{eq:CobordismCondition4} in
Theorem \ref{thm:Cobordism}, that
$\delta\equiv -w^2-3\chi_h(X)\pmod 4$
and $\delta-2m\ge 0$, respectively,
follow from \eqref{eq:deltaAndmAssumption}.
Thus, Lemma \ref{lem:BlownUpCobordism} yields
\begin{equation}
\begin{aligned}
\label{eq:CompareCoeff3}
D^w_X(h^{\delta-2m}x^m)
&=
\sum_{\begin{subarray}{l}i+j+2k\\=\delta-2m\end{subarray}}
\sum_{K\in B'(X)}
(-1)^{{\eps(w,K)}}
\frac{2(i+1)SW'_X(K)}{\delta-2m+1}
\tilde b_{i+1,j,k}(K\cdot\La)
\\
{}&\qquad\times
\langle K,h\rangle^i
\langle \La,h\rangle^j
Q_X(h)^k.
\end{aligned}
\end{equation}
We now verify that we can apply Propositions \ref{prop:HighDegreeCoefficients} and \ref{prop:LowICoefficientDifferenceRelation}
and Corollary \ref{cor:PolynomialDependence}
to compute the coefficients $\tilde b_{i+1,j,k}$ in \eqref{eq:CompareCoeff3}.
The indices $i,j,k,m$ appearing in
\eqref{eq:CompareCoeff3} satisfy
\begin{equation}
\label{eq:IndexConstraint}
i+1+j+2k+2m=\delta+1.
\end{equation}
To match the notation of  Propositions \ref{prop:HighDegreeCoefficients} and \ref{prop:LowICoefficientDifferenceRelation},
we will write the first two arguments of the coefficients in \eqref{eq:CoefficientAbbreviation} as
\begin{equation}
\label{eq:Defineq}
q:=\chi_h(X),
\end{equation}
and $c_1^2(X)-1= q-3-n$, where
\begin{equation}
\label{eq:Definen}
n:= \chi_h(X)-c_1^2(X)-2=c(X)-2.
\end{equation}
The definitions \eqref{eq:Defineq} and  \eqref{eq:Definen},
the property that $b^+\ge 3$ for standard manifolds,
and our earlier observation that we can assume $c(X)\ge 5$
imply that
\begin{equation}
\label{eq:FirstTwoArgsAssump}
q\ge 2\quad\text{and}\quad n\ge 2,
\end{equation}
as required in
Propositions \ref{prop:HighDegreeCoefficients} and \ref{prop:LowICoefficientDifferenceRelation}.

We now verify the hypotheses of
Proposition \ref{prop:HighDegreeCoefficients} for the coefficients $\tilde b_{i+1,j,k}$ in \eqref{eq:CoefficientAbbreviation}
with $i\ge c(X)-3$.  The
condition \eqref{eq:HighDegreeCondition1} holds because $i+1\ge c(X)-2=n$ by \eqref{eq:Definen}.
In the notation of Proposition \ref{prop:HighDegreeCoefficients} for $\tilde b_{i+1,j,k}$,
we have $A=i+1+j+k+2m$ and so $A=\delta+1$ by \eqref{eq:IndexConstraint}.
The property
\eqref{eq:La2Assumption} of $\La^2$
and \eqref{eq:Defineq}
imply that
\begin{equation}
\label{eq:La2Bound1}
\La^2>\delta-c(X)-4q=\delta-n-2-4q=A-n-3-4q,
\end{equation}
so condition \eqref{eq:HighDegreeCondition2} holds.
The condition $A\ge 2m$ in \eqref{eq:HighDegreeCondition3} holds by \eqref{eq:deltaAndmAssumption}.
Our choice of $\La=2\La_0$ implies that $\La^2\equiv \La\cdot K\equiv 0\pmod 2$ for all
$K\in B(X)$, and thus condition
\eqref{eq:HighDegreeCondition4}
holds as well,
noting that $x = \Lambda^2$ and $y=\Lambda\cdot K$.
Hence, Proposition \ref{prop:HighDegreeCoefficients} and the equality $A=\delta+1$ imply that, for all $i\ge c(X)-3$,
\begin{equation}
\label{eq:HighDegreeCoeffInFinalSum}
\tilde b_{i+1,j,k}(\chi_h(X),c_1^2(X)-1,K\cdot\La,\La^2,m)
=
\begin{cases}
\displaystyle\frac{(\delta+1-2m)!}{k!(i+1)!}2^{m-k-c(X)+2}  & \text{if $j=0$,}
\\
0 & \text{if $j>0$.}
\end{cases}
\end{equation}
We now verify the hypotheses of Proposition \ref{prop:LowICoefficientDifferenceRelation}
and Corollary \ref{cor:PolynomialDependence} hold.
Observe that $i+1\le c(X)-3=n-1$ by \eqref{eq:Definen}, so condition \eqref{eq:DiffRelatCond0} holds.
The inequality in \eqref{eq:La2Bound1} implies that \eqref{eq:DiffRelatCond1} holds.
Because $A=\delta+1\equiv c(X)+1\equiv n+3\pmod 4$ by
\eqref{eq:deltaAndmAssumption} and \eqref{eq:Definen}, the fact that $\La^2=(2\La_0)^2\equiv 0\pmod 4$
implies that
$$
\La^2\equiv 0\equiv A-(n+3)\pmod 4,
$$
and thus condition \eqref{eq:DiffRelatCond2} holds.
We already showed that condition \eqref{eq:HighDegreeCondition4} holds and that implies
condition \eqref{eq:DiffRelatCond3} holds.  Therefore, Proposition
\ref{prop:LowICoefficientDifferenceRelation}
applies to the coefficients $\tilde b_{i+1,j,k}$ with $i\le c(X)-3$.
The hypotheses of Corollary  \ref{cor:PolynomialDependence}
are those of Proposition
\ref{prop:LowICoefficientDifferenceRelation} and
the conditions we have previously verified in \eqref{eq:CorPolDepAssumptions}.
Thus, Corollary  \ref{cor:PolynomialDependence}
implies that the coefficients $\tilde b_{i+1,j,k}$ with 
$i\le c(X)-4$ can be written as
\begin{equation}
\label{eq:LowDegCoeffExp}
\begin{aligned}
{}&\tilde b_{i+1,j,k}(\chi_h(X),c_1^2(X)-1,K\cdot\La,\La^2,m)
\\
{}&\quad=
\sum_{u=0}^{c(X)-4-i}
\tilde b_{u,i+1,j,k}(q,q-n-3,\La^2,m)\langle K,h_\La\rangle^u,
\end{aligned}
\end{equation}
where $h_\La=\PD[\La]\in H_2(X;\RR)$.

We now abbreviate,
$$
\tilde b_{u,i+1,j,k}:=\tilde b_{u,i+1,j,k}(q,q-n-3,\La^2,m),
$$
and split the sum on the right-hand-side of \eqref{eq:CompareCoeff3} into two parts:
\begin{equation}
\begin{aligned}
\label{eq:CompareCoeff3a}
D^w_X(h^{\delta-2m}x^m)
&=
\sum_{\begin{subarray}{l}i+j+2k\\=\delta-2m,\\ i\le c(X)-4\end{subarray}}
\sum_{K\in B'(X)}
(-1)^{{\eps(w,K)}}
\frac{2(i+1)SW'_X(K)}{\delta-2m+1}
\\
{}&\qquad\quad\times
\sum_{u=0}^{c(X)-4-i}
\tilde b_{u,i+1,j,k}
\langle K,h\rangle^i\langle K,h_\La\rangle^u
\langle \La,h\rangle^j
Q_X(h)^k
\\
{}&\quad +
\sum_{\begin{subarray}{l}i+j+2k\\=\delta-2m,\\ i\ge c(X)-3\end{subarray}}
\sum_{K\in B'(X)}
(-1)^{{\eps(w,K)}}
\frac{2(i+1)SW'_X(K)}{\delta-2m+1}
\\
{}&\qquad\quad\times
\tilde b_{i+1,j,k}(K\cdot\La)
\langle K,h\rangle^i
\langle K,h\rangle^j
Q_X(h)^k.
\end{aligned}
\end{equation}
Because the coefficients $\tilde b_{u,i+1,j,k}$ do not depend on $\La\cdot K$,
we can rewrite the first sum on the right-hand-side of \eqref{eq:CompareCoeff3a} as
\begin{equation}
\label{eq:CompareCoeff3b}
\begin{aligned}
{}&
\sum_{\begin{subarray}{l}i+j+2k\\=\delta-2m,\\ i\le c(X)-4\end{subarray}}
\frac{2(i+1)SW'_X(K)}{\delta-2m+1}
\langle \La,h\rangle^j
Q_X(h)^k
\\
{}&\qquad\times
\sum_{u=0}^{c(X)-4-i}
\tilde b_{u,i+1,j,k}
\sum_{K\in B'(X)}
(-1)^{{\eps(w,K)}}SW'_X(K)
\langle K,h\rangle^i\langle K,h_\La\rangle^u.
\end{aligned}
\end{equation}
By \eqref{eq:ParityOfVanishingbuCoeff} and the equality
$n\equiv c(X)$ from \eqref{eq:Definen},
\begin{equation}
\label{eq:VanishingCoeffByParity}
\tilde b_{u,i+1,j,k}=0 \quad\text{if $u\equiv c(X)+i+1\pmod 2$.}
\end{equation}
We now consider the terms in the sum \eqref{eq:CompareCoeff3b}
with $u\equiv n+i\pmod 2$.
For $u$ and $i$ satisfying $0\le u+i\le c(X)-4$, and $u\equiv n+i\pmod 2$, and $w\in H^2(X;\ZZ)$ characteristic,
Lemma \ref{lem:SCSTVanishingSum}
implies that
$$
\sum_{K\in B'(X)}
(-1)^{{\eps(w,K)}}SW'_X(K)
\langle K,h\rangle^i\langle K,h_\La\rangle^u
=0.
$$
Because $0\le u\le c(X)-4-i$ and thus $0\le u+i\le c(X)-4$
for all terms in the sum \eqref{eq:CompareCoeff3b},
the preceding equality and \eqref{eq:VanishingCoeffByParity} imply that
the sum  \eqref{eq:CompareCoeff3b} vanishes.

Hence, the terms in the sum  \eqref{eq:CompareCoeff3a} with $i\le c(X)-4$ vanish.
By employing that fact and the formula \eqref{eq:HighDegreeCoeffInFinalSum} for the coefficients $\tilde b_{i+1,j,k}$, we can rewrite \eqref{eq:CompareCoeff3a} as
\begin{align*}
{}&D^w_X(h^{\delta-2m}x^m)
\\
{}&\quad=
\sum_{\begin{subarray}{l}i+2k\\=\delta-2m,\\ i\ge c(X)-3\end{subarray}}
\sum_{K\in B'(X)}
(-1)^{{\eps(w,K)}}
\frac{2(i+1)SW'_X(K)}{\delta-2m+1}
\\
{}&\qquad\times
\frac{(\delta-2m+1)!}{k!(i+1)!2^{k+c(X)-2-m}}
\langle K,h\rangle^i
Q_X(h)^k
\\
{}&\quad=
\sum_{\begin{subarray}{l}i+2k\\=\delta-2m,\\ i\ge c(X)-3\end{subarray}}
\sum_{K\in B'(X)}
(-1)^{{\eps(w,K)}}
SW'_X(K)
\frac{(\delta-2m)!}{k!i!2^{k+c(X)-3-m}}
\langle K,h\rangle^i
Q_X(h)^k.
\end{align*}
Comparing this equality with \eqref{eq:DInvarForWCB'Sum}
in Lemma \ref{lem:ReduceDFormToB'Sum} and observing that
the terms in \eqref{eq:DInvarForWCB'Sum} with $i\le c(X)-4$ also vanish by
the superconformal simple type property
shows that Witten's Conjecture \ref{conj:WC} holds.
\end{proof}

\begin{rmk}
\label{rmk:CoeffAmbiguity}
The proof of Theorem \ref{thm:SCSTImpliesWC} also illustrates the limits of the method of
applying Lemma \ref{lem:AlgCoeff} to examples of four-manifolds
satisfying Witten's Conjecture \ref{conj:WC} to determine the coefficients $\tilde b_{i,j,k}$.

We can see that if $X$ has superconformal simple type,
then  by Lemma \ref{lem:SCSTVanishingSum}, changing the coefficients $\tilde b_{u,i,j,k}$ in
\eqref{eq:CompareCoeff3a} would not change the expression for the Donaldson invariant
given by the cobordism formula because the expression in \eqref{eq:CompareCoeff3b}
would still vanish.
Thus, applying Lemma \ref{lem:AlgCoeff} to an equality of the
form \eqref{eq:CompareCoeff3a}, on a manifold of superconformal
simple type, does not determine the coefficients $\tilde b_{i,j,k}$.

Because  all standard four-manifolds have
superconformal simple type by \cite{FL8}, this
makes it unlikely that one could extract
more information about the coefficients $\tilde b_{i,j,k}$ by applying this
method to other four-manifolds satisfying Witten's Conjecture \ref{conj:WC}.
\end{rmk}

\begin{proof}[Proof of Corollary \ref{cor:NonZeroIntImpliesWC}]
The result follows immediately from Theorem \ref{thm:SCSTImpliesWC}
and the result in \cite{FL8} that standard four-manifolds
satisfying the hypotheses of Corollary \ref{cor:NonZeroIntImpliesWC} have
superconformal simple type.
\end{proof}

\bibliography{/Users/pfeehan/Dropbox/LATEX/Bibinputs/master}
\bibliographystyle{amsplain-nodash}

\end{document}